\documentclass[a4paper,reqno,10pt]{amsart}

\usepackage{amsmath,amsthm, amssymb}
\usepackage{mathrsfs}
\usepackage[shortlabels]{enumitem}
\usepackage{graphicx}
\usepackage[font=small]{caption}
\usepackage{xcolor}
\usepackage{url}

\newcommand{\N}{\mathbb{N}}
\newcommand{\R}{{\mathbb{R}}}
\newcommand{\C}{{\mathbb{C}}}
\newcommand{\Z}{{\mathbb{Z}}}
\newcommand{\dd}{{{\rm d}}}
\newcommand{\ii}{{\rm i}}

\newcommand{\cf}{\emph{cf.}}
\newcommand{\ie}{{\emph{i.e.}}}
\newcommand{\eg}{{\emph{e.g.}}}

\newcommand{\ov}{\overline}

\newcommand{\eps}{\varepsilon}

\renewcommand{\H}{{\mathcal{H}}}

\newcommand{\spp}{\sigma_{\rm p}}

\newcommand{\Dom}{{\operatorname{Dom}}}

\newcommand{\Ran}{{\operatorname{Ran}}}

\newcommand{\Rank}{{\operatorname{rank}}}
\renewcommand{\Re}{\operatorname{Re}}
\renewcommand{\Im}{\operatorname{Im}}
\newcommand{\dist}{\operatorname{dist}}

\newcommand{\supp}{\operatorname{supp}}
\newcommand{\Tr}{\operatorname{Tr}}

\newcommand{\Num}{\operatorname{Num}}

\newcommand{\BigO}{\mathcal{O}}

\newcommand{\lspan}{{\operatorname{span}}}

\theoremstyle{plain}

\newtheorem{theorem}{Theorem}[section]
\newtheorem{lemma}[theorem]{Lemma}

\newtheorem{proposition}[theorem]{Proposition}
\newtheorem{corollary}[theorem]{Corollary}

\theoremstyle{definition}

\newtheorem{remark}[theorem]{Remark}

\newcommand\cH{\mathcal H}

\newcommand\cM{\mathcal M}

\newcommand\cS{\mathcal S}

\numberwithin{equation}{section}

\theoremstyle{definition}
\newtheorem{asmR}{Assumption}

\begin{document}

\title[Local form-subordination condition]{Local form-subordination condition and Riesz basisness of root systems}

%    author one information
\author{Boris Mityagin}
\address[Boris Mityagin]{
Department of Mathematics,
The Ohio State University,
231 West 18th Ave,
Columbus, OH 43210, USA}
%\curraddr{}
\email{mityagin.1@osu.edu}
%\thanks{}

%    author two information
\author{Petr Siegl}
\address[Petr Siegl]{Mathematical Institute, University of Bern, Alpeneggstr. 22, 3012 Bern, Switzerland \& On leave from Nuclear Physics Institute CAS, 25068 \v Re\v z, Czech Republic}
%\curraddr{}
\email{petr.siegl@math.unibe.ch}
%\thanks{}

\subjclass[2010]{47A55, 47A70, 34L10, 34L40}

\keywords{non-self-adjoint operators, Riesz basis, Schr\"odinger operators with complex and singular potentials}

\date{31st July 2016}

\dedicatory{In memory of our colleague and friend Michael Solomyak.}

\begin{abstract}
We exploit the so called form-local subordination in the analysis of non-symmetric perturbations of unbounded self-adjoint operators with isolated simple positive eigenvalues. If the proper condition relating the size of gaps between the unperturbed eigenvalues and the strength of perturbation, measured by the form-local subordination, is satisfied, the root system of the perturbed operator contains a Riesz basis and usual asymptotic formulas for perturbed eigenvalues and eigenvectors hold. The power of the abstract perturbation results is demonstrated particularly on Schr\"odinger operators with possibly unbounded or singular complex potential perturbations.
\end{abstract}

\thanks{
We acknowledge the support of the University of Bern (B.M., October-November 2015) and the Ohio State University (P.S., May 2016) for our visits there. The research of P.S. is supported by the \emph{Swiss National Foundation} Ambizione grant No.~PZ00P2\_154786.}

\maketitle

\section{Introduction}

Let $T$ be a Schr\"odinger operator in $L^2(\R)$
\begin{equation}\label{T.Schr.intro}
T = -\frac{\dd^2}{\dd x^2} + Q(x) + V(x)
\end{equation}
where $Q$ is a sufficiently regular \emph{real} single-well potential behaving as $|x|^\beta$, $\beta \geq 1$, at infinity and $V$ is a \emph{complex}, possibly unbounded or singular, perturbation. The spectrum of $T$ is discrete under mild restrictions on $V$, guaranteeing basically that $V$ is indeed a ``small'' perturbation of the self-adjoint operator 
\begin{equation}\label{A.Schr.intro}
A = -\frac{\dd^2}{\dd x^2} + Q(x).
\end{equation}
However, it is quite unclear under which conditions on $V$ the eigensystem of $T$ shares the good properties of the eigenfunctions of $A$, forming an orthonormal basis of $L^2(\R)$. More specifically, for which $V$ does the eigensystem of $T$ contain a Riesz basis?

Notice that the Riesz basisness of eigensystem is in particular strongly linked with the spectral stability/instability (pseudospectra/resolvent norm far from spectrum) of the spectrum of $T$. The spectral instability is well-known to occur for Schr\"odinger operators with complex potentials like the rotated oscillator of Davies
\begin{equation}\label{HO.rot}
-\frac{\dd^2}{\dd x^2} + \ii x^2,
\end{equation}
for which the eigensystem does not contain even a basis. Such results are obtained in several works and are typically based on the norm estimates of the resolvent (pseudospectra) or eigenprojections, see \eg~\cite{Davies-1999-200, Davies-2000-32, Davies-2004-70, Henry-2012-350, Henry-2014-4, Henry-2014-15, Siegl-2012-86,Krejcirik-2015-56}. 
Nonetheless, by proving the unboundedness the eigenprojection norms in \cite{Mityagin-2013arx}, no basis result follows  also for the shifted oscillator 
\begin{equation}\label{HO.shift}
-\frac{\dd^2}{\dd x^2} + x^2 + \ii x,
\end{equation}
where $Q(x)=x^2$ clearly dominates the imaginary perturbation $V(x)=\ii x$. On the other hand, it was showed in \cite{Mityagin-2016-106} that the eigensystem of
\begin{equation}\label{HO.delta}
-\frac{\dd^2}{\dd x^2} + x^2 + \ii \delta(x),
\end{equation}
does contain a Riesz basis; the latter holds also for an infinite number of $\delta$'s, namely for $\sum_{k \in \Z} \nu_k \delta(x-x_k)$ with $\nu \in \ell^1(\Z)$ and $\{x_k\} \subset \R$, see \cite{Mityagin-2016-106,Mityagin-2015-54,Mityagin-toappear}. The perturbations in \eqref{HO.shift} and \eqref{HO.delta} seem to be of a different nature, nevertheless, their strength is the \emph{same} if described in classical terms of relative boundedness or $p$-subordination. More specifically, when viewed in the sense of quadratic forms,
\begin{equation}
b_1[\psi] \equiv b_{\ii x}[\psi] = \ii \int_{\R} x |\psi(x)|^2 \, \dd x, \qquad b_2[\psi] \equiv b_{\ii \delta(x)}[\psi] = \ii |\psi(0)|^2,
\end{equation}
both $b_1$ and $b_2$ are $p$-subordinated with $p=1/2$ to the form $a$, associated with the self-adjoint harmonic oscillator. Namely, there is $C>0$ such that, for all $\psi \in \Dom(a)= \{\psi \in W^{1,2}(\R): x \psi(x) \in L^2(\R) \}$,
\begin{equation}
|b_i[\psi]| \leq C a[\psi]^p \|\psi\|^{2(1-p)}, \quad i = 1,2, \quad p = \frac12, 
\end{equation}
where
\begin{equation}
a[\psi] = \|\psi'\|^2 + \int_\R x^2 |\psi(x)|^2 \, \dd x. 
\end{equation}
These two examples clearly demonstrate that the classical sufficient conditions for the Riesz basisness of the eigensystem based on $p$-subordination, see \eg~\cite[Thm.XIX.2.7]{DS}  or \cite[Thm.6.12]{Markus-1988}, cannot provide satisfactory answers for 
\eqref{T.Schr.intro}. 

The objective of this paper is to analyze deeper, both on the abstract level and focused on \eqref{T.Schr.intro}, the perturbation problems by employing a condition that measures the strength of the perturbation in a more subtle way than the ordinary $p$-subordination. In detail, we work in the following setting.

Let $A$ be a self-adjoint operator with compact resolvent in a separable Hilbert space $\H$. Let the eigenvalues $\{\mu_k\}_{k\in \N}$ of $A$ be positive, eventually simple and satisfy
\begin{equation}\label{asm:A}
\begin{aligned}
& A \psi_k = \mu_k \psi_k, \quad \|\psi_k\| =1,
\\
& \exists \gamma>0, \ \exists\kappa>0, \ \exists N_0> 0, \ \forall k \geq N_0, \quad
\mu_{k+1} -\mu_k \geq  \kappa k^{\gamma-1};
\end{aligned}
\end{equation}
by $\{\psi_k\}$ we denote normalized eigenvectors of $A$ related to $\{\mu_k\}$. The key assumption on the form $b$, representing the perturbation, is the so-called \emph{local form-subordination condition}
\begin{equation}\label{asm:b}
\exists \alpha \in \R, \ 2\alpha + \gamma > 1, \quad \exists M_b >0, \quad \forall m,n \in \N, \quad  |b(\psi_m,\psi_n)| \leq \frac{M_b}{m^{\alpha} n^{\alpha}}.
\end{equation}

The main abstract result, Theorem~\ref{thm:RB}, states that if conditions \eqref{asm:A} and \eqref{asm:b} are satisfied, then the eigensystem of the perturbed operator $T$ contains a Riesz basis; the operator $T$ is defined via the form sum $a+b$, where is $a$ associated with $A$, see Section~\ref{sec:def.op} for details on introducing $T$. Moreover, the eigenvalues of $T$ are eventually simple, the usual asymptotic formulas for the corrections of $\{\mu_k\}$ and $\{\psi_k\}$ hold and remainder estimates, expressed in terms of $\alpha$ and $\gamma$, are given, see~Theorem~\ref{thm:ev.asym}. 

The applicability of the conditions \eqref{asm:A} and \eqref{asm:b} is demonstrated on the Schr\"odinger operator $T$ from \eqref{T.Schr.intro} viewed as a perturbation of the self-adjoint $A$ from \eqref{A.Schr.intro}. The condition \eqref{asm:A} is satisfied for this $A$ with $\gamma = 2\beta/(\beta+2)$, see Proposition~\ref{prop:Q.gaps}. On the other hand, for $\beta \geq 2$, the form $b$ generated by the potential $V$ satisfies the condition~\eqref{asm:b} if \eg~$V = V_1 + V_2 + V_3 + V_4$ where
\begin{equation}\label{A.V.intro}
\begin{aligned}
&\exists \eps >0, \quad  |x|^{\frac{2-\beta}{2} +\eps} V_1(x) \in L^\infty(\R),
\\ 
&\exists p \in [1,\infty), \quad V_2 \in L^p(\R),
\\
&\exists s \in [0 ,\frac{\beta-1}{2 \beta}), \quad V_3  \in W^{-s,2}(\R),  
\\
&\exists \{\nu_k\} \in \ell^1(\Z), \quad \exists \{x_k\} \subset \R, \quad V_4  = \sum_{k \in \Z} \nu_k \, \delta(x-x_k),   
\end{aligned}
\end{equation}
see Theorems~\ref{thm:aho.Lptau}, \ref{thm:aho.sing} and Corollaries~\ref{cor:aho.Lptau}, \ref{cor:aho.sing} for precise statements (with $\beta < 2$ allowed). Moreover, for $V \in L^1(\R)$ with a controlled decay at infinity, see Theorem~\ref{thm:aho.ev}, we prove that the first corrections of $\mu_k$ read (with the precisely determined constant $C_\beta$)
\begin{equation}\label{ev.asym.intro}
\lambda_n^{(1)} = C_\beta \, n^{-\frac{2}{\beta+2}} \int_\R V(x) \; \dd x + o \left(n^{-\frac{2}{\beta+2}}\right), \quad n \to \infty. 
\end{equation}

Although our main motivation are Schr\"odinger operators $T$ from \eqref{T.Schr.intro}, the abstract results are applicable to other problems. In particular, see Section~\ref{subsec:finite.band} for infinite finite band matrices and Section~\ref{subsec:Neumann} for perturbations of $-\dd^2/\dd x^2$ on a finite interval with Neumann boundary conditions. For the latter, some classical results, \eg~the Riesz basisness of the eigensystem for the separated boundary conditions, see \cite[Chap.XIX.3]{DS}, follow immediately when formulated in our setting. The efficiency of our approach can be further illustrated on that the amount of effort needed to prove the Riesz basisness for  $-\dd^2/\dd x^2$ from above perturbed \eg~by the infinite number of complex $\delta$-interactions (which can hardly be treated by ODE methods) is basically the same as when perturbing by a bounded potential; see Section~\ref{subsec:Neumann}. As $-\dd^2/\dd x^2$ on $(-1,1)$ with Dirichlet boundary conditions is a limit of $A$ from \eqref{A.Schr.intro} with $Q(x)=|x|^\beta$ for $\beta \to +\infty$, it is not surprising that by taking formally $\beta = + \infty$ in the formulas, \eg~\eqref{ev.asym.intro}, and conditions derived for $A$ with $\beta \in (1,\infty)$, we recover those for the limit $-\dd^2/\dd x^2$ on $(-1,1)$, see~Remark~\ref{rem:beta.inf}.

Regarding the relations to previous works, the special version of the condition \eqref{asm:b} with $\gamma=1$ was introduced in \cite{Mityagin-2016-106}; the relation to the operator version of \eqref{asm:b} used in \cite{Adduci-2012-10, Adduci-2012-73, Shkalikov-2010-269, Shkalikov-2012-18} is discussed in \cite{Mityagin-2016-106} as well. Comparing to previous papers, we allow here a faster condensation of $\{\mu_k\}$ at infinity, namely $\mu_k \sim k^\gamma$ with $\gamma >0$ is possible, \cf~\cite{Adduci-2012-73} with the restriction $\gamma > 1/2$. At the same time, the proof on the main abstract Theorem~\ref{thm:RB} on the Riesz basisness is simplified by using only the equivalent condition for a Riesz basis, see~\cite[Thm.VI.2.1]{Gohberg-1969}, together with the Schur test for infinite matrices, see \eg~\cite{Schur-1911-140}, \cite[Thm.5.2]{Halmos-1978-96}; thereby we avoid the Kato's lemma on projections \cite[Thm.V.4.17a]{Kato-1966} used in \cite{Adduci-2012-10, Adduci-2012-73,Mityagin-2016-106}. Moreover, the condition \eqref{asm:b} is sharp in the sense that it cannot be weakened to  $2\alpha + \gamma =1$, see Section~\ref{subsec:opt}. As for Schr\"odinger operators $T$, the form-local subordination allows for treating more singular potentials (in $L^p$ with $1 \leq p < 2$) than in \cite{Adduci-2012-73} and also the distributional ones. The asymptotic formula for the first eigenvalue correction, see \eqref{ev.asym.intro} or \eqref{lam.1.dec}, should be compared with precise two terms asymptotics of eigenvalues in \cite{Mityagin-2015-54,Mityagin-toappear} for the harmonic oscillator ($Q(x)=x^2$) perturbed by $\delta$-interactions (instead of $L^1$-potentials).

The paper is organized as follows. In Section \ref{sec:def.op} we recall the definition of the perturbed operator $T$ based on the form sum. Moreover, using the classical perturbation theory, we derive basic properties of $T$, in particular the completeness of its eigensystem. Main abstract results, Riesz basisness, asymptotic formulas for eigenvalues and eigenvectors, are stated and proved in Section~\ref{sec:abstract}. In Section~\ref{sec:tech.lem}, we prove several technical lemmas used in the proofs in Section~\ref{sec:abstract}. Section~\ref{sec:ex.sim} consists of several examples, showing the optimality of our assumptions and applicability of the main theorems in simpler examples. Finally, Section~\ref{sec:aho} is devoted to the analysis of the Sch\"odinger operators \eqref{T.Schr.intro}.

%
%\begin{equation}
%\|B \psi_n\| = o(n^{\gamma-1}), \quad n \to \infty,
%\end{equation}
%

\section{Preliminaries}
\label{sec:def.op}

The definition of the (abstract) perturbed operator $T$ is based on representation and perturbations theorems from \cite[Sec.VI]{Kato-1966}. Let $A$ be a self-adjoint operator satisfying \eqref{asm:A} and $b$ be a form satisfying \eqref{asm:b}. The operator $A$ is associated with the form
\begin{equation}
a[f]:=a(f,f) = \|A^\frac12 f\|^2, \quad \Dom(a)=\Dom(A^\frac12),
\end{equation}
see~\cite[Thm.VI.2.23]{Kato-1966} for the second representation theorem. We show below that the form $t:=a+b$ is sectorial and closed. Thus, using the first representation theorem, see~\cite[Thm.VI.2.1]{Kato-1966}, $t$ is associated with the unique m-sectorial operator $T$, which is our perturbed operator. In particular, the Schr\"odinger operator $T$ from \eqref{T.Schr.intro} with $V \in L^1_{\rm loc}(\R)$ is obtained by choosing $A$ as in \eqref{A.Schr.intro} and $b$ as the form $b_V$ generated by $V$, see~\eqref{bV.def}; for the distributional potentials see~\eqref{bV.sing}.

To show that $t$ is sectorial and closed, notice that when the condition~\eqref{asm:b} is satisfied, the form $b$ is $p$-subordinated to $a$ with some $p=p(\alpha,\gamma) \in [0,1)$, 
\begin{equation}\label{form.sub}
\forall f \in \Dom(a), \quad |b[f]| \leq C (a[f] )^p \|f\|^{2(1-p)};
\end{equation}
see Lemma~\ref{lem:subord} for details. Hence $b$ is relatively bounded with respect to $a$ with the bound $0$ in particular, i.e. 
\begin{equation}\label{b.rb}
\forall \eps >0, \  \exists C_\eps \geq 0, \ \forall f \in \Dom(a), \quad |b[f]| \leq \eps a[f] + C_\eps \|f\|^2.
\end{equation}
Thus the form $t=a+b$ is indeed sectorial and closed and it determines uniquely the m-sectorial operator $T$ with compact resolvent, see~\cite[Thm.VI.3.4]{Kato-1966}. Moreover, the norm of the resolvent of $T$ decays along every ray except $\R^+$, namely,
\begin{equation}\label{T.res.decay}
\forall \vartheta \in (0,2\pi), \quad \lim_{t \to + \infty}\|(e^{\ii \vartheta}t - T)^{-1} \| = 0;
\end{equation}
the proof is based on the relative boundedness with the bound $0$ and
\begin{equation}
\|(z - T)^{-1} \| \leq \frac{1}{\dist(z, \ov{\Num(T)})},
\end{equation}
where $\Num(T)$ denotes the numerical range of $T$, see~\cite[Thm.V.3.2]{Kato-1966}.

The operator $T$ can be also described as follows. We introduce the operators 
\begin{equation}\label{K.def}
K(z) f := \sum_{k \in \N} (z-\mu_k)^{-\frac 12} \langle f, \psi_k \rangle \psi_k, \quad z \in \rho(A),
\end{equation}
where, for $0\neq w \in \C$ and $s\in \R$, the $s$-power of $w$ is taken as $w^s := |w|^s e^{\ii s \arg w}$ with $-\pi < \arg w \leq \pi$. Notice that
\begin{equation}\label{K^2}
K(z)^2 = (z-A)^{-1}, \quad z \in \rho(A).
\end{equation}
Then the operator $T$ reads
\begin{equation}\label{T.def}
T = A^{\frac 12}(I-B(0))A^{\frac 12},
\end{equation}
where $B(z)$, $z \in \rho(A)$, is the operator uniquely determined by the bounded form 
\begin{equation}
b(K(z) \cdot,K(z)^* \cdot).
\end{equation}
In more detail, $B(z)$ is determined from the relation $\langle B(z)f,g \rangle = b(K(z) f,(K(z)^* g) $ for all $f,g \in \H$. 
For all $z \in \rho(A)$, we have
\begin{equation}
z-T = K(z)^{-1}(I-B(z)) K(z)^{-1}
\end{equation}
and this relation yields a suitable factorization of the resolvent of $T$, namely
\begin{equation}\label{Tz.res.dec}
(z-T)^{-1} = K(z)(I-B(z))^{-1}K(z),
\end{equation}
provided $I-B(z)$ is invertible and $z \in \rho(A)$; see also~\cite[Lemma~1]{Agranovich-1994-28}, \cite[Chap.VI.3.1]{Kato-1966}.

Points $z$ for which $I-B(z)$ is invertible certainly exist since we get straightforwardly from~\eqref{b.rb} and~\eqref{K^2} that
\begin{equation}
\forall \eps >0, \  \exists \tilde C_\eps \geq 0, \ \forall z<0, \quad  \|B(z)\| 
\leq 
\eps + \frac{\tilde C_\eps}{|z|}.  
\end{equation}
Thus $\|B(z)\|<1$ if $z<0$ and $|z|$ is sufficiently large, hence $I-B(z)$ is invertible for such $z$'s. 

The inequality \eqref{asm:A} implies that, for sufficiently large $k$, we have
\begin{equation}\label{muk.gr}
\exists c > 0, \ \exists k_0 \in \N, \ \forall k > k_0, \quad \mu_k \geq c k^\gamma, 
\end{equation}
see Lemma~\ref{lem:sum.dist} for details. Thus for all $z \in \rho(A)$, $K(z)$ is in the Schatten class $\cS_{2r}$ with any $r$ satisfying $r \gamma >1$. The factorization~\eqref{Tz.res.dec} and the fact that $(I-B(z))^{-1}$ is bounded for some $z<0$ implies that 
\begin{equation}\label{T.Sr}
\forall z \in \rho(T), \ \forall r> \frac 1 \gamma, \quad (z-T)^{-1} \in \cS_r.
\end{equation}
Combining \eqref{T.Sr}, \eqref{T.res.decay} and~\cite[Cor.XI.9.31]{DS}, we obtain the following.

\begin{proposition}\label{prop:complete}
Let conditions~\eqref{asm:A}, \eqref{asm:b} hold and let $T$ be as in~\eqref{T.def}. Then the eigensystem of $T$ is complete in $\H$.
\end{proposition}

Finally, we remark that $T^*$, the adjoint of $T$, is the operator associated with the adjoint form $t^*=a+b^*$, see~\cite[Chap.VI.1.1]{Kato-1966} and \cite[Thm.VI.2.5]{Kato-1966}. We have from the definition of the adjoint form and \eqref{asm:b} that
\begin{equation}\label{b*.cond}
|b^*(\psi_m,\psi_n)| = |\overline{b(\psi_n,\psi_m)}| \leq \frac{M_b}{m^\alpha n^\alpha}.
\end{equation}
So the analogues of results derived for $T$ under assumptions~\eqref{asm:A} and~\eqref{asm:b} are valid also for $T^*$, in particular, the eigensystem of $T^*$ is complete.

\section{Abstract perturbation results}
\label{sec:abstract}
\subsection{Localization of eigenvalues of $T$}
For $N \in \N$ and $h>0$, we define sets 
\begin{equation}\label{Pin.def}
\begin{aligned}
\Pi_0&=\Pi_0(N,h):=\{ z \in \C: - h < \Re z <  \mu_N + \frac \kappa 2 N^{\gamma-1}, |\Im z| <  h\}, \\
\Pi_k&:=\left\{ 
z \in \C: \mu_k - \frac \kappa 2 (k-1)^{\gamma-1}  < \Re z < \mu_k + \frac \kappa 2 k^{\gamma-1}, |\Im z| < \frac \kappa 2 k^{\gamma-1}  
\right\},
\\ \Pi&:=\bigcup_{k \in \N} \Pi_k, \quad \Gamma_0:=\partial \Pi_0, \quad \Gamma_k :=\partial \Pi_k, \quad k \in \N.
\end{aligned}
\end{equation}

\begin{proposition}\label{prop:loc}
Let conditions \eqref{asm:A}, \eqref{asm:b} hold and let $T$ be as in \eqref{T.def}. Then there exist $N>N_0$ and $h>0$ such that 
\begin{equation}\label{T.loc}
\sigma(T) = \spp(T) \subset \Pi_0(N,h) \cup \bigcup_{k > N} \Pi_k,
\end{equation}
where $\Pi_k$ are as in \eqref{Pin.def}. Moreover, with these $N$ and $h$, Riesz projections 
\begin{equation}\label{Pn.SN.def}
\begin{aligned}
S_N:=\frac{1}{2\pi \ii} \int_{\Gamma_0}(z-T)^{-1} \dd z, \quad 
P_n  := \frac{1}{2\pi \ii} \int_{\Gamma_n}(z-T)^{-1} \dd z, \quad n>N,
\end{aligned}
\end{equation}
are well-defined and 
\begin{equation}\label{rank.SN.Pn}
\Rank \, S_N =  N, \qquad \Rank \,  P_n =  1, \quad n>N.
\end{equation}
\end{proposition}
\begin{proof}
Our first aim is to find $N$ and $h$ such that $(z-T)^{-1}$ is bounded for all $z \notin \Pi_0(N,h) \cup (\cup_{k > N} \Pi_k)$. With the help of the resolvent factorization \eqref{Tz.res.dec}, it suffices to prove that $\|B(z)\| \leq 1/2$ for all such $z$. Let $f = \sum_{j=1}^{\infty} f_j \psi_j \in \H$, then
\begin{equation}\label{loc.est.1}
\begin{aligned}
\|B(z)f\|^2  &= 
\sum_{k=1}^{\infty} |\langle B(z) f, \psi_k \rangle|^2 
= 
\sum_{k=1}^{\infty} 
\left| 
\sum_{j=1}^{\infty} \frac{f_j b(\psi_j,\psi_k) }{(z-\mu_j)^{\frac{1}{2}} (z-\mu_k)^{\frac{1}{2}}}
\right|^2 
\\
&\leq 
M_b^2 \sum_{k=1}^{\infty} 
\frac{1}{k^{2\alpha}|\mu_k-z|}
\left(
\sum_{j=1}^{\infty} \frac{|f_j|}{ j^{\alpha}|\mu_j-z|^{\frac{1}{2}}}
\right)^2 
\\
&
\leq 
M_b^2 
\left(
\sum_{k=1}^{\infty} 
\frac{1}{k^{2\alpha}|\mu_k -z|}
\right)^2
\|f\|^2.
\end{aligned}
\end{equation}
From Lemma \ref{lem:sum.zn}, see \eqref{sigma.n} for the definition of $\sigma_{\omega,\gamma}$, we obtain
\begin{equation}
\sup_{\substack{z \notin \Pi \\[1mm] \Re z \geq \mu_n}}\sum_{k=1}^{\infty} 
\frac{1}{k^{2\alpha}|\mu_k -z|} = \BigO(\sigma_{2\alpha,\gamma}(n)), \quad n \to \infty.
\end{equation}
Hence, we can indeed choose $N>N_0$ such that, for all $z \notin \Pi$ and $\Re z \geq \mu_N$, we have $\|B(z)\| \leq 1/2$. 

The next step is the estimate of $\|B(z)\|$ for $\Re z \leq - h_1$. Splitting the final sum in \eqref{loc.est.1} as in \eqref{sum.ag.n.split}--\eqref{sum.N0.est}, we have
\begin{equation}
\label{h1.est}
\|B(z)\|
\leq 
M_b 
\left(
\frac{\max\{N_0,N_0^{1-2\alpha}\}}{h_1}
+ \sum_{k=N_0+1}^{\infty} \frac{1}{ k^{2\alpha}(\mu_k+h_1)}
\right).
\end{equation}
It follows from Lemma~\ref{lem:sum.dist}, see also \eqref{muk.gr}, and $2\alpha+\gamma>1$ that 
\begin{equation}
\sum_{k=N_0+1}^{\infty} \frac{1}{k^{2\alpha}\mu_k} < \infty,
\end{equation}
thus 
\begin{equation}
\lim_{h_1 \to \infty} \sum_{k=N_0+1}^{\infty} \frac{1}{k^{2\alpha}(\mu_k+h_1)} = 0.
\end{equation}
Hence there is $h_1 >0$ such that, for all $z$ with $\Re z \leq- h_1$, we have $\|B(z)\| \leq 1/2$.

In the third step, we estimate $\|B(z)\|$ for $z$ with $\Re z \in [ - h_2, \mu_N + \frac \kappa 2 N^{\gamma-1}]$ and $|\Im z| \geq h_2$. As in \eqref{h1.est}, we get the estimate
\begin{equation}
\begin{aligned}
\|B(z)\|
&\leq 
M_b 
\left(
\frac{C}{h_2}
+ \sum_{k=2 N}^{\infty} \frac{1}{ k^{2\alpha} \sqrt{(\mu_k-\mu_{N+1})^2 + h_2^2}}
\right)
\end{aligned}
\end{equation}
and conclude that 
\begin{equation}
\lim_{h_2 \to \infty} \sum_{k=2 N}^{\infty} \frac{1}{k^{2\alpha} \sqrt{(\mu_k-\mu_{N+1})^2 + h_2^2}} = 0.
\end{equation}
Thus we can choose $h_2>0$ such that $\|B(z)\| \leq 1/2$ for all $z$ with $\Re z \in [ - h_2, \mu_N + \frac \kappa 2 N^{\gamma-1}]$ and $|\Im z| \geq h_2$. 

In summary, taking $h:=\max\{h_1,h_2\}$, we have that $\|B(z)\| \leq 1/2$ for all $z \notin \Pi_0(N,h) \cup (\cup_{k > N} \Pi_k)$, thus \eqref{T.loc} is proved.

The standard argument, based on \cite[Lemma~VII.6.7]{DS}, shows that 
\begin{equation*}
\Tr \, \frac{1}{2\pi \ii} \int_{\Gamma_n} (z-A)^{-\frac 12}(I-t B(z))^{-1} (z-A)^{-\frac 12} \, \dd z, \quad 0 \leq t \leq 1,   
\end{equation*}
is a continuous integer-valued function, thus it is constant and \eqref{rank.SN.Pn} follows.
\end{proof}

\subsection{Asymptotics of eigenvalues and eigenvectors of $T$}

It follows from Proposition \ref{prop:loc} that the eigenvalues $\{\lambda_n\}$ of $T$ become eventually simple (for $n > N$) and localized around those of the unperturbed operator $A$. The rest of the spectrum is localized in $\Pi_0$.
Therefore for $n >N$, we have $\Tr P_n = 1$, see \eqref{Pn.SN.def}, thus 
\begin{equation}\label{EV.formula}
\begin{aligned}
\lambda_n - \mu_n &= \frac{1}{2  \pi \ii} \Tr \int_{\Gamma_n} z ((z-T)^{-1} - (z-A)^{-1}) \, \dd z
\\
& = \frac{1}{2  \pi \ii} \Tr \int_{\Gamma_n} (z-\mu_n) ((z-T)^{-1} - (z-A)^{-1}) \, \dd z, \quad n >N.
\end{aligned}
\end{equation}
As in \eg~\cite{Elton-2003,Elton-2004-16,Mityagin-2014a,Mityagin-2015-54,Mityagin-toappear}, the further analysis of $\lambda_n$ with $n>N$ relies on the formula \eqref{EV.formula}. 
The eigenvectors $\{\phi_n\}_{n >N}$ of $T$, satisfying $T\phi_n = \lambda_n \phi_n$, are found using
\begin{equation}\label{EF.formula}
\begin{aligned}
\phi_n &= \frac{1}{2  \pi \ii} \int_{\Gamma_n} (z-T)^{-1} \psi_n \, \dd z 
\\
 &= \psi_n + \frac{1}{2  \pi \ii} \int_{\Gamma_n} ( (z-T)^{-1} -(z-A)^{-1})\psi_n \, \dd z 
, \quad n>N.
\end{aligned}
\end{equation}
%

%??? 
%Examples in Section~\ref{sec:ex.opt} show that the remainder estimate in \eqref{lambda.n.exp} as well as \eqref{EV.cor.est}, based on an absolute value estimates, are optimal in the abstract setting. Nevertheless, it is showed in \cite{Mityagin-2014a,Mityagin-2015-54}, that such absolute value estimates can be too crude when analyzing particular (differential) operators. 
%???

%
\begin{theorem}\label{thm:ev.asym}
Let conditions \eqref{asm:A}, \eqref{asm:b} hold and let $T$ be as in \eqref{T.def}. Denote by $\{\lambda_n\}_{n\in \N}$ the eigenvalues of $T$ so that $\lambda_n \in \Pi_n$ for $n >N$.
Then
\begin{equation}\label{lambda.n.exp}
\lambda_n = \mu_n + \sum_{k=1}^j\lambda_n^{(k)} + r_n^{(j)}, \quad j \in \N, \quad n >N,
\end{equation}
where 
\begin{equation}\label{lambda.n.exp.terms}
\lambda_n^{(k)} =  \frac{1}{2  \pi \ii}  \Tr \int_{\Gamma_n} (z-\mu_n)K(z) B(z)^k K(z) \, \dd z, \quad k \geq 1,
\end{equation}
and
\begin{equation}\label{EV.cor.est}
|\lambda_n^{(k)}| = \BigO \left(n^{\gamma-1} \sigma_{2\alpha,\gamma}(n)^k \right),
\qquad 
|r_n^{(j)}| = \BigO \left(\frac{\sigma_{2\alpha,\gamma}(n)^{j+1}}{n^{2 \alpha}} \right), \quad n \to \infty.
\end{equation}
Moreover, the vectors
\begin{equation}\label{phi_n.def}
\phi_n := \frac{1}{2  \pi \ii} \int_{\Gamma_n} (z-T)^{-1} \psi_n \, \dd z 
= \psi_n + \sum_{k=1}^j{\phi_n^{(k)}} + \rho_n^{(j)}, \quad j \in \N, \quad n >N,
\end{equation}
where
\begin{equation}\label{phi_n.def.terms}
\phi_n^{(k)} = \frac{1}{2  \pi \ii} \int_{\Gamma_n} K(z) B(z)^k K(z) \psi_n \, \dd z
\end{equation}
and
\begin{equation}\label{EF.cor.est}
\|\phi_n^{(k)}\| = \BigO(\sigma_{2\alpha,\gamma}(n)^k), \quad \|\rho_n^{(j)}\| = \BigO(\sigma_{2\alpha,\gamma}(n)^{j+1}), \quad n \to \infty,
\end{equation}
satisfy (with some $N_1 \geq N$)
\begin{equation}\label{phi_n.EF}
T \phi_n = \lambda_n \phi_n, \quad \phi_n \neq 0, \quad n > N_1 \geq N.
\end{equation}
\end{theorem}

\begin{proof}%
If $n >N$ and $z \in \Gamma_n$, we have $\|B(z)\|<1$, see Proposition \ref{prop:loc} and its proof. Thus we can expand $(I-B(z))^{-1}$ into a convergent sum and thereby get
\begin{equation}\label{T.rez.exp}
(z-T)^{-1} - (z-A)^{-1} = K(z) \left(\sum_{m=1}^\infty B(z)^m \right) K(z).
\end{equation}
Inserting~\eqref{T.rez.exp} into \eqref{EV.formula} and \eqref{EF.formula}, we obtain~\eqref{lambda.n.exp}, \eqref{lambda.n.exp.terms} and~\eqref{phi_n.def}, \eqref{phi_n.def.terms}. The claim \eqref{phi_n.EF} follows from \eqref{phi_n.def} and \eqref{EF.cor.est}.

The main part of the proof is the explanation that the operators in~\eqref{lambda.n.exp.terms} are indeed trace class and that the estimates \eqref{EV.cor.est}, \eqref{EF.cor.est} hold. In what follows, we denote the $r$-Schatten class norm by $\|\cdot\|_r$. 
The estimates are done mostly in the same way as in~\cite[Sec.2.4]{Mityagin-2015-54}.

We start with $K(z)B(z)K(z)$, so $k=1$ in \eqref{lambda.n.exp.terms} and \eqref{EV.cor.est}. Using 
\begin{equation}\label{B.psi}
B(z)\psi_j = \sum_{m=1}^{\infty} \langle B(z) \psi_j, \psi_m \rangle \psi_m,
\end{equation}
we obtain
\begin{equation}
K(z)B(z)K(z) \psi_j = \sum_{m=1}^\infty \frac{b(\psi_j,\psi_m)}{(z-\mu_j)(z-\mu_m)} \psi_m, \quad z \in \Gamma_n.
\end{equation}
The integration leads to 
\begin{equation}\label{KBK.int}
\begin{aligned}
&\frac{1}{2  \pi \ii}  \int_{\Gamma_n}  (z-\mu_n) \langle K(z) B(z) K(z) \psi_j, \psi_l \rangle \, \dd z
\\
& \quad = 
\frac{1}{2  \pi \ii}  \int_{\Gamma_n}  \frac{(z-\mu_n)b(\psi_j,\psi_l)  }{(z-\mu_j)(z-\mu_l)}  \, \dd z 
 = b(\psi_n,\psi_n) \delta_{j,n} \delta_{l,n}.
\end{aligned}
\end{equation}
Thus the operator in \eqref{lambda.n.exp.terms} reads
\begin{equation}\label{KBK.op}
\frac{1}{2  \pi \ii}  \int_{\Gamma_n}  (z-\mu_n)  K(z) B(z) K(z)  \, \dd z
 = b(\psi_n,\psi_n) \langle \cdot, \psi_n \rangle, 
\end{equation}
and so it is of rank one and hence trace class. Moreover, $|\lambda_n^{(1)}| = \BigO(n^{-2\alpha})$, $n \to \infty$, so \eqref{EV.cor.est} holds for $k=1$.

Next we show that $B(z) \in \cS_2$ for all $z \in \Gamma_n$, $n >N$, and
\begin{equation}
\sup_{ z \in \Gamma_n} \|B(z)\|_2 = \BigO(\sigma_{2\alpha,\gamma}(n)), \quad n \to \infty.
\end{equation}
Indeed, using again \eqref{B.psi}, for $z \in \Gamma_n$, $n >N$, we obtain
\begin{equation}\label{B.HS}
\begin{aligned}
\|B(z)\|_2^2 &= \sum_{m=1}^\infty |\langle B(z) \psi_m, B(z) \psi_m \rangle |^2 
= 
\sum_{j,m=1}^\infty 
\frac{|b(\psi_m,\psi_j)|^2}{|z-\mu_j| |z-\mu_m|}
\\
& \leq 
M_b^2 \left(\sum_{m=1}^\infty \frac{1}{m^{2\alpha}|z-\mu_m| } \right)^2 
= \BigO(\sigma_{2\alpha,\gamma}(n)^2), \quad n \to \infty;
\end{aligned}
\end{equation}
in the last step, we use Lemma~\ref{lem:sum.zn}. Hence we have $K(z) B(z)^k K(z) \in \cS_1$ for  $k \geq 2$ and all $z \in \Gamma_n$ with $n>N$. From $|\Tr C| \leq \|C\|_1$, see~\cite[Cor.XI.9.8]{DS}, $\|K(z)\|^2 = 1/|z-\mu_n|$ if $z \in \Gamma_n$, \eqref{B.HS}, $|\Gamma_n|=\BigO(n^{\gamma-1})$ as $n \to \infty$ and \cite[Lemma~XI.9.14]{DS}, we have
\begin{equation}
\begin{aligned}
|\lambda_n^{(k)}| & \leq \frac{1}{2  \pi }  \int_{\Gamma_n}  |z-\mu_n| \|K(z)\|^2 \|B(z)\|_2^k \, |\dd z|
\\&
= \BigO(n^{\gamma-1}\sigma_{2\alpha,\gamma}(n)^k), \quad n \to \infty.
\end{aligned}
\end{equation}
Moreover, for the remainder $r_n^{(j)}$, $j \geq 2$, we have
\begin{equation}
\begin{aligned}
|r_n^{(j)}| & \leq \sum_{m=j+1}^\infty \frac{1}{2  \pi }  \int_{\Gamma_n}  |z-\mu_n| \|K(z)\|^2 \|B(z)\|_2^m \, |\dd z|
\\
& = \BigO(n^{\gamma-1}\sigma_{2\alpha,\gamma}(n)^{j+1}), \quad n \to \infty, 
\end{aligned}
\end{equation}
thus~\eqref{EV.cor.est} is proved.
Finally for the eigenvectors, we get from \eqref{phi_n.def.terms}, \eqref{B.HS} that
\begin{equation}
\begin{aligned}
\|\phi_n^{(k)}\| &\leq \frac{1}{2  \pi }  \int_{\Gamma_n}  \|K(z)\|^2 \|B(z)\|_2^k \, |\dd z|
= \BigO(\sigma_{2\alpha,\gamma}(n)^{k}), \quad n \to \infty,
\\
\|\rho_n^{(j)}\|  &\leq \sum_{m=j+1}^\infty \frac{1}{2  \pi }   \int_{\Gamma_n}  \|K(z)\|^2 \|B(z)\|_2^m \, |\dd z| 
= \BigO(\sigma_{2\alpha,\gamma}(n)^{j+1}), \quad n \to \infty,
\end{aligned}
\end{equation}
so \eqref{EF.cor.est} is proved as well. 
\end{proof}

\begin{remark}\label{rem:ev.cor}
In particular, we have
\begin{equation}\label{EV.cor}
\begin{aligned}
\lambda_n^{(1)} &= b(\psi_n,\psi_n), \qquad 
\lambda_n^{(2)} = \sum_{j=1,j\neq n}^{\infty} \frac{|b(\psi_n,\psi_j)|^2}{\mu_n-\mu_j},
\\
\phi_n^{(1)} &= \sum_{j=1,j\neq n}^{\infty} \frac{b(\psi_n,\psi_j)}{\mu_n-\mu_j} \psi_j.
\end{aligned}
\end{equation}
\end{remark}
\begin{proof}
Formulas for $\lambda_n^{(2)}$ and $\phi_n^{(1)}$ follow from \eqref{lambda.n.exp.terms} and \eqref{phi_n.def.terms}, respectively, by the calculation of residues. To derive the formula for $\lambda_n^{(1)}$, we can use~\eqref{KBK.op}; like \eg~ in~\cite[Lemmas~8, 9]{Djakov-2013-255}, it is important to integrate before taking the trace (or norm in the proof of Theorem~\ref{thm:ev.asym}) since all but one term in \eqref{KBK.int} are zero after the integration. 
\end{proof}

\subsection{Riesz basis property of the eigensystem}

Proposition \ref{prop:loc} shows that there are only finitely many eigenvalues of $T$ in $\Pi_0$, namely $\{\lambda_n\}_{n=1}^{N'}$, $N' \leq N$, with algebraic multiplicities $\{m_n\}_{n=1}^{N'}$, $\sum_{n=1}^{N'} m_n = N$. The remaining eigenvalues $\{\lambda_n\}_{n>N}$ are simple. Hence the eigensystem of $T$ contains at most a finite number of root vectors associated with $\{\lambda_n\}_{n=1}^{N'}$ and the rest consists of eigenvectors related to  $\{\lambda_n\}_{n>N}$. For $n>N_1$, these eigenvectors can be selected as $\{\phi_n\}_{n>N_1}$ from Theorem~\ref{thm:ev.asym}. 

\begin{theorem}\label{thm:RB}
Let conditions \eqref{asm:A}, \eqref{asm:b} hold, let $T$ be as in \eqref{T.def}, let $N_1 > N$ be such that \eqref{phi_n.EF} holds for all $n>N_1$ and let $S_{N_1}$ be as in \eqref{Pn.SN.def} with $N=N_1$. Then the set $\{\phi_n\}_{n=1}^\infty$, where $\{\phi_n\}_{n=N_1}^\infty$ is a basis of $\Ran(S_{N_1})$ and $\phi_n$ with $n>N_1$ are as in \eqref{phi_n.def},
is a Riesz basis of $\cH$. 
\end{theorem}

\begin{proof}
The proof is based on \cite[Thm.VI.2.1]{Gohberg-1969} and Schur test for infinite matrices, see \eg~\cite{Schur-1911-140}, \cite[Thm.5.2]{Halmos-1978-96}. We need to verify that $\{\phi_n\}$ is complete in $\cH$, there exists a complete system $\{\tilde \phi_n\}$ that is biorthogonal to $\{\phi_n\}$ and we have
\begin{equation}\label{RB.cond.f}
\forall f \in \cH, \quad \sum_{n=1}^\infty |\langle f, \phi_n \rangle|^2 < \infty, \quad \sum_{n=1}^\infty |\langle f, \tilde \phi_n \rangle|^2 < \infty.
\end{equation}

The system $\{\phi_n\}$ is complete by Proposition~\ref{prop:complete}. As the biorthogonal system, we can select vectors $\tilde \phi_n$ from the eigensystem of $T^*$ with 
\begin{equation}
\tilde \phi_n := \frac{1}{\langle \phi_n^*,\phi_n \rangle} \phi_n^*, \quad \phi_n^* := \frac{1}{2  \pi \ii} \int_{\Gamma_n} (z-T^*)^{-1} \psi_n \, \dd z, \quad n >N_1.
\end{equation}
Due to~\eqref{b*.cond}, we obtain as in Theorem~\ref{thm:ev.asym} that
\begin{equation}\label{phi_n*.def}
\phi_n^* = \psi_n + \sum_{k=1}^j{{\phi_n^*}^{(k)}} + {\rho_n^*}^{(j)}, \quad j \in \N,
\end{equation}
where
\begin{equation}\label{phi_n*.def.terms}
{\phi_n^*}^{(k)} = \frac{1}{2  \pi \ii} \int_{\Gamma_n} K(z) (B(\ov z)^*)^k K(z) \psi_n \, \dd z
\end{equation}
and
\begin{equation}\label{EF.cor.est*}
\|{\phi_n^*}^{(k)}\| = \BigO(\sigma_{2\alpha,\gamma}(n)^k), \quad \|{\rho_n^*}^{(j)}\| = \BigO(\sigma_{2\alpha,\gamma}(n)^{j+1}), \quad n \to \infty;
\end{equation}
notice that then $\langle \phi_n^*,\phi_n \rangle = 1 + \BigO(\sigma_{2\alpha,\gamma}(n))$ as $n\to \infty$. Moreover, the system $\{\tilde \phi_n\}$ is complete in $\cH$; see remarks below Proposition~\ref{prop:complete}. 

The crucial step is to show \eqref{RB.cond.f}. We analyze the sum with $\{\phi_n\}$ only. In view of \eqref{b*.cond} and \eqref{phi_n*.def}--\eqref{EF.cor.est*}, the reasoning for the second sum is completely analogous.
Clearly, it suffices to consider the sum for $n > N_1$ only. We give the detailed proof for the case $2\alpha \leq 1$, the other case is similar.

First select $j \in \N$ such that $2(j+1)(2\alpha+\gamma-1)>1$, then, using \eqref{EF.cor.est} and \eqref{sigma.n}, we get 
\begin{equation}\label{f.rho}
\sum_{n=N_1+1}^\infty |\langle f , \rho_n^{(j)} \rangle|^2 
\leq 
\|f\|^2 \sum_{n=N_1+1}^\infty \|\rho_n^{(j)} \|^2 < \infty.
\end{equation}
Thus it remains to estimate $|\langle f, \phi_n^{(k)}\rangle|^2$ for $k=1,\dots,j$. We first derive that
\begin{equation}
B(z)^k \psi_{j_0} = \sum_{j_1,\dots,j_k} 
\left(
\prod_{l=1}^k \langle B(z) \psi_{j_{l-1}}, \psi_{j_l} \rangle
\right) \psi_{j_k}
\end{equation}
and hence from \eqref{phi_n.def.terms}
\begin{equation}\label{phi_n.repr}
\phi_n^{(k)} = 
\frac{1}{2  \pi \ii} \int_{\Gamma_n} 
\sum_{j_1,\dots,j_k} 
\frac{
 b(\psi_n,\psi_{j_1})
\prod_{l=2}^k b(\psi_{j_{l-1}}, \psi_{j_l})}
{(z-\mu_n) \prod_{l=1}^k(z-\mu_{j_l})}
 \psi_{j_k}
\, \dd z
\end{equation}
Decomposing $f = \sum_{m=1}^{\infty} f_m \psi_m$, we obtain from \eqref{phi_n.repr} and \eqref{asm:b} that
\begin{equation}\label{f.phi_n^k}
\begin{aligned}
| \langle f, \phi_n^{(k)} \rangle | 
& \leq 
M_b^k  
\frac{1}{2\pi} \int_{\Gamma_n} 
\sum_{j_1,\dots,j_k}
\frac{|f_{j_k}| \, |\dd z|}
{|z-\mu_n||z-\mu_{j_k}| n^\alpha j_k^\alpha \prod_{l=1}^{k-1} |z-\mu_{j_l}| j_l^{2\alpha} }
\\
& \leq M_b^k 
\left(
\sum_{j=1}^\infty 
\frac{1}{j^{2\alpha}|z_n-\mu_j|} 
\right)^{k-1}
\sum_{m=1}^\infty 
\frac{|f_m| }
{|z_n-\mu_m| m^\alpha n^\alpha } \times
\\ & \quad
\frac{1}{2\pi} \int_{\Gamma_n} \frac{|\dd z|}{|z_n-\mu_n|},
\end{aligned}
\end{equation}
where $z_n \in \Gamma_n$ is such that the maximum of the integrand in the first integral in \eqref{f.phi_n^k} is attained; notice that $z_n$ depends on $f$. From Lemma~\ref{lem:sum.est.2} and  $|\Gamma_n|= \BigO(n^{\gamma-1})$ we get further that there is a constant $C>0$ such that
\begin{equation}\label{f.phi_n^k.2}
\begin{aligned}
| \langle f, \phi_n^{(k)} \rangle |^2 
& \leq C \sigma_{2\alpha,\gamma}(n)^{2(k-1)}
 \left(\sum_{m=1}^\infty 
\frac{|f_m| }
{|z_n-\mu_m| m^\alpha n^\alpha }
\right)^2.
\end{aligned}
\end{equation}
The final step is to estimate the sum of $|\langle f, \phi_n^{(k)} \rangle |^2$ for $n>N_1$. Since $2\alpha+\gamma-1>0$, see also \eqref{sigma.n}, it suffices to consider the case $k=1$ only. For the latter, we get 
\begin{equation}\label{f.phi.M}
\sum_{n=N_1+1}^{\infty} | \langle f, \phi_n^{(1)} \rangle |^2   
\leq C \sum_{n=1}^{\infty}
\left( 
\sum_{m=1}^{\infty} \frac{ |f_m|}{m^\alpha n^{\alpha} |z_n- \mu_m|} 
\right)^2
= 
C \| \mathcal M \tilde f\|^2_{\ell^2(\N)}
\end{equation}
where $\mathcal{M}$ is an operator acting in $\ell^2(\N)$ with matrix elements 
\begin{equation}
\begin{aligned}
\cM_{mn}&=
\frac{1}{m^\alpha n^{\alpha}|z_n-\mu_m|}, & m,n \in \N
\end{aligned}
\end{equation}
and $\tilde{f}=\{|f_m|\} \in \ell^2(\N)$. To estimate $\|\cM\|$ we employ the Schur test. By applying Lemma \ref{lem:sum.zn} and its slight modification for the second sum, we get 
\begin{equation}
\begin{aligned}
\sum_{m=1}^{\infty}|\mathcal{M}_{mn}| \frac{1}{m^\alpha} 
&=
\sum_{m=1}^{\infty} \frac{1}{n^{\alpha}m^{2\alpha}|z_n-\mu_m|}
\leq \frac{C_1}{n^\alpha} \frac{\log e n}{n^{2\alpha+\gamma-1}}
\leq \frac{C_2}{n^\alpha}, 
\\
\sum_{n=1}^{\infty}|\mathcal{M}_{mn}| \frac{1}{n^\alpha} 
&=
\sum_{n=1}^{\infty} \frac{1}{n^{2\alpha}m^{\alpha}|z_n-\mu_m|}
\leq \frac{C_3}{m^\alpha} \frac{\log e m}{m^{2\alpha+\gamma-1}}
\leq \frac{C_4}{m^\alpha}, 
\end{aligned}
\end{equation}
thus the Schur test yields that $\|\cM\| < \infty$. The latter, \eqref{f.phi.M}, \eqref{f.phi_n^k.2} and \eqref{f.rho} show that \eqref{RB.cond.f} holds for $\{\phi_n\}$. 
\end{proof}

\begin{remark}
Let conditions \eqref{asm:A}, \eqref{asm:b} hold and let, in addition, 
\begin{equation}\label{asm:Bari}
\begin{cases}
2\alpha +\gamma > \frac32 & \text{if} \quad \alpha \leq \frac 12,
\\[2mm]
\gamma > \frac12 & \text{if} \quad \alpha > \frac 12.
\end{cases}
\end{equation}
Then the system $\{\phi_n\}_{n=1}^\infty$ from Theorem~\ref{thm:RB} is a Bari basis, namely
\begin{equation}\label{phi_n.Bari}
\sum_{n=1}^\infty \|\psi_n-\phi_n\|^2 < \infty.
\end{equation}

\end{remark}
\begin{proof}
Inserting \eqref{phi_n.def} into \eqref{phi_n.Bari}, we infer that, for $n>N_1$, 
\begin{equation}
\|\psi_n-\phi_n\|^2 \leq 2 (\|\phi_n^{(1)}\|^2 +  \|\rho_n^{(1)}\|^2) 
=\BigO(\sigma_{2\alpha,\gamma}(n)^2), \quad n \to \infty. 
\end{equation}
Thus the conditions \eqref{asm:Bari} imply that \eqref{phi_n.Bari} holds by~\eqref{sigma.n}.
\end{proof}

\section{Technical lemmas}
\label{sec:tech.lem}

\begin{lemma}\label{lem:sum.dist}
Let $\{\mu_k\}_{k\in\N}$ satisfy \eqref{asm:A} and let $j,k \in \N$, $k>j>N_0$. Then
\begin{equation}
\mu_k-\mu_j \geq 
\frac{\kappa}{\gamma}
\begin{cases}
(k-1)^\gamma - (j-1)^\gamma, & \gamma \geq 1,
\\[2mm]
k^\gamma - j^\gamma, & 0 < \gamma < 1.
\end{cases}
\end{equation}
Thus, 
\begin{equation}
\mu_k-\mu_j \geq 
\frac{\kappa}{\gamma}
\left(
(k-1)^\gamma - j^\gamma
\right), \quad \gamma >0.
\end{equation}

\end{lemma}
\begin{proof}
From \eqref{asm:A}, we get
\begin{equation}
\begin{aligned}
\mu_k-\mu_j &= \sum_{i=j+1}^{k} (\mu_i - \mu_{i-1})
\geq  \kappa \sum_{i=j+1}^{k} (i-1)^{\gamma-1}
\\
&
\geq \kappa 
\begin{cases}
\int_j^k (x-1)^{\gamma-1} \dd x, & \gamma \geq 1,
\\[2mm]
\int_j^k x^{\gamma-1} \dd x, & 0 < \gamma < 1.
\end{cases}
\end{aligned}
\end{equation}
The proof is concluded by the direct integration and simple manipulations.
\end{proof}

\begin{lemma}\label{lem:subord}
Let $A$ and $b$ satisfy~\eqref{asm:A} and~\eqref{asm:b}.
Then for every $\tau \in (0,2\alpha+\gamma-1)$, there exists $C>0$ such that, with $p=\max \{0,1-\tau/\gamma\}$, 
\begin{equation}\label{form.sub.det}
\forall f \in \Dom(a), \quad |b[f]| \leq C (a[f] )^p \|f\|^{2(1-p)}.
\end{equation}
\end{lemma}
\begin{proof}
We write $f= \sum_{j=1}^{\infty} f_j \psi_j$. Using Lemma~\ref{lem:sum.dist}, see also~\eqref{muk.gr}, we can verify that the following lower bound for $a[f]$ holds with some $C_1>0$
\begin{equation}\label{a.le}
\begin{aligned}
a[f] = \sum_{j=1}^{\infty} \mu_j |f_j|^2 
\geq C_1 \sum_{j=1}^{\infty} j^\gamma |f_j|^2.
\end{aligned}
\end{equation}
Using the condition~\eqref{asm:b} and H\"older inequality, we get (with $2(\alpha+\beta)>1$, $\gamma > 2 \beta$)
\begin{equation}\label{b.ue}
\begin{aligned}
|b[f]| & = 
\left| 
\sum_{j,k=1}^{\infty} f_j \overline f_k  b(\psi_j,\psi_k)  \right| 
\leq 
M_b 
\left(
\sum_{j=1}^{\infty} \frac{|f_j|}{j^{\alpha}} 
\right)^2
= 
M_b 
\left(
\sum_{j=1}^{\infty} |f_j|j^{\beta} \frac{1}{j^{\alpha+\beta}}  
\right)^2
\\ 
& 
\leq  
M_b \left(\sum_{j=1}^{\infty} j^{2\beta} |f_j|^2 \right)
\sum_{j=1}^{\infty} \frac{1}{j^{2(\alpha+\beta)}}
%\\
%&
\leq 
C_2 \left(\sum_{j=1}^{\infty} j^\gamma |f_j|^2 \right)^\frac{2\beta}{\gamma} \|f\|^{2\left(1-\frac{2\beta}{\gamma}\right)}.
\end{aligned}
\end{equation}
The inequality in~\eqref{form.sub.det} follows by combining~\eqref{a.le} and~\eqref{b.ue} and putting $2\beta:=\gamma-\tau$ when $\alpha \leq 1/2$ and $\beta:=0$ when $\alpha>1/2$.
\end{proof}

\begin{lemma}\label{lem:sum.est.2}
Let $n \in \N$, $\gamma >0$ and $\omega + \gamma >1$. Then
\begin{equation}\label{sum.est}
\sum_{k =1, k \neq n}^{\infty} \frac{1}{k^{\omega}|k^\gamma-n^\gamma| } 
= \BigO(\sigma_{\omega,\gamma}(n)),
\end{equation}
where 
\begin{equation}\label{sigma.n}
\sigma_{\omega,\gamma}(n):=
\begin{cases}
n^{-\omega-\gamma+1} \log e n, & \omega \leq 1, 
\\
n^{-\gamma},  & \omega > 1.
\end{cases}
\end{equation}

\end{lemma}
\begin{proof}
The absolute constant is denoted by the letter $C$ and can vary from line to line. In all estimates below, we assume that $n \in \N$ is sufficiently large. Clearly,
\begin{equation}\label{eq.lem.1}
\sum_{\substack{k =1 \\ k\neq n} }^{\infty} \frac{1}{k^{\omega}|k^\gamma-n^\gamma| }  \leq 
\sum_{k=1}^{n-1} \frac{1}{k^{\omega}(n^\gamma-k^\gamma) } 
+ \sum_{k=n+2}^{\infty} \frac{1}{k^{\omega}(k^\gamma-n^\gamma) }
+ \BigO(n^{-\omega-\gamma+1}).
\end{equation}
The first term on the right of \eqref{eq.lem.1} can be estimated as
\begin{equation}
\begin{aligned}
\sum_{k=1}^{n-1} \frac{1}{k^{\omega}(n^\gamma-k^\gamma) } 
&= 
\left( \sum_{k=1}^{[\frac n2]} + \sum_{k=[\frac n2] }^{n-1} \right) 
\frac{1}{k^{\omega}(n^\gamma-k^\gamma)}
%\\
%& \leq 
%\frac{1}{n^\gamma-[\frac n2]^\gamma} \sum_{k=1}^{[\frac n2]} \frac{1}{k^{\omega}} 
%+ \frac{C}{n^\omega} \sum_{k=[\frac n2] }^{n-1}
%\frac{1}{n^\gamma-k^\gamma}
\\
& \leq C \left( 
\frac{1}{n^\gamma} \sum_{k=1}^{n} \frac{1}{k^{\omega}} + \frac{1}{n^\omega} \sum_{k=[\frac n2]}^{n-1}
\frac{1}{n^\gamma-k^\gamma}
\right)
\end{aligned}
\end{equation}
and similarly the second term of \eqref{eq.lem.1} as
\begin{equation}
\begin{aligned}
\sum_{k=n+2}^{\infty} \frac{1}{k^{\omega}(k^\gamma-n^\gamma) } & = 
\left( \sum_{k=n+2}^{2n}  + \sum_{k=2n+1}^{\infty} \right) \frac{1}{k^{\omega}(k^\gamma-n^\gamma)}
\\
& \leq C \left(
\frac{1}{n^\omega} \sum_{k=n+2}^{2n} \frac{1}{k^\gamma-n^\gamma}
+ \sum_{k=2n+1}^{\infty} \frac{1}{k^{\omega + \gamma}}
\right).
\end{aligned}
\end{equation}
For a monotone, continuous, non-negative function $f$ in interval $[a,b]$, $a,b \in \N$, $a<b$, we have
\begin{equation}\label{sum.int.est}
\sum_{i=a}^b f(i) \leq f(a) + f(b) + \int_a^b f(x) \, \dd x;
\end{equation}
$f(a)$ can be omitted if $f$ is increasing and similarly $f(b)$ can be omitted if $f$ is decreasing.
Thus applying \eqref{sum.int.est}, we get
\begin{align}
\sum_{k=1}^{n} \frac{1}{k^{\omega}}
&\leq
1 + \frac{1}{n^\omega} + \int_1^n \frac{\dd x }{x^\omega} 
\leq
C \begin{cases}
n^{1-\omega}, & \omega <1, \\
\log n, & \omega =1, \\
1, & \omega >1,
\end{cases}
\\
\sum_{k=2n+1}^{\infty} \frac{1}{k^{\omega + \gamma}}
& \leq
\frac{1}{(2n+1)^{\omega+\gamma}} + 
\int_{2n+1}^{\infty}  \frac{\dd x}{x^{\omega + \gamma}}
\leq  
\frac{C}{n^{\omega+\gamma-1}}.
\end{align}
Moreover, since $(1-y)/(1-y^\gamma) \to 1/\gamma$ as $y\to 1$, we obtain

\begin{align}\label{n2n.sum.2}
\sum_{k=[\frac n2] }^{n-1} \frac{1}{n^\gamma-k^\gamma} &
\leq 
\frac{1}{n^{\gamma}-(n-1)^\gamma} +  \int_{\frac n2-1}^{n-1} \frac{\dd x}{n^\gamma - x^\gamma} 
\\
& \leq  \frac{C}{n^{\gamma-1}}  \left( 
1 + \int_{\frac 12-\frac1n}^{1-\frac 1n} \frac{\dd y}{1 - y^\gamma}
\right) 
\leq C \frac{\log n}{n^{\gamma-1}},
\\
\sum_{k=n+2}^{2n} \frac{1}{k^\gamma-n^\gamma}&
\leq
\frac{C}{n^{\gamma-1}}  \left(
1+ \int_{1+\frac 2n}^{2} \frac{\dd y}{y^\gamma-1}
\right) 
\leq C \frac{\log n}{n^{\gamma-1}}.
\end{align}
Combing all the inequalities above, we receive \eqref{sum.est}.
\end{proof}

\begin{lemma}\label{lem:sum.zn}
Let	the conditions \eqref{asm:A} and $\omega+\gamma>1$ hold. Then 
\begin{equation}\label{sum.ag.n}
\sup_{\substack{z \notin \Pi \\[1mm] \Re z \geq \mu_n} } \ \sum_{k=1}^{\infty} 
\frac{1}{k^{\omega}|\mu_k -z|} = \BigO(\sigma_{\omega,\gamma}(n)), \quad n \to \infty,
\end{equation}
where $\sigma_{\omega,\gamma}(n)$ is as in \eqref{sigma.n} and $\Pi$ as in \eqref{Pin.def}.
\end{lemma}
\begin{proof}
Define sets in $\C$
\begin{equation}\label{Xi.j.def}
\Xi_j:=\left\{ z \notin \Pi \, : \Re z \in [\mu_{j-1}+\frac \kappa 2(j-1)^{\gamma-1},\mu_{j+1}-\frac \kappa 2 j^{\gamma-1}] \right\}, \quad j \in \N;
\end{equation}
note that we can cover the region $\{z \notin \Pi \,: \, \Re z \geq \mu_n\}$ by $\cup_{j \geq n} \Xi_j$. In all estimates below, we assume that $n$ is sufficiently large, in particular $n > N_0+3$.

As we do not have information on $\{\mu_k\}_{k=1}^{N_0}$, we split the sum in \eqref{sum.ag.n},
\begin{equation}\label{sum.ag.n.split}
\sum_{k=1}^{\infty} 
\frac{1}{k^{\omega}|\mu_k -z|}
=	\left(
	\sum_{k=1}^{N_0} 
	+ 
	\sum_{k=N_0+1}^{\infty} 
	\right)
	\frac{1}{k^{\omega}|\mu_k -z|}
\end{equation}
and do a rough estimate of the first finite sum. Namely, using Lemma \ref{lem:sum.dist} in the last step, we get
\begin{equation}\label{sum.N0.est}
\sup_{\substack{z \notin \Pi \\[1mm] \Re z \geq \mu_n}} \, \sum_{k=1}^{N_0} 
\frac{1}{k^{\omega}|\mu_k -z|}
\leq
\frac{\max\{N_0,N_0^{1-\omega}\}}{\mu_n-\mu_{N_0}} = \BigO(n^{-\gamma}), \quad n \to \infty.
\end{equation}
The second sum in \eqref{sum.ag.n.split} is estimated with the help of Lemmas \ref{lem:sum.dist} and \ref{lem:sum.est.2},
\begin{equation}\label{sum.inf.est}
\begin{aligned}
&\sup_{\substack{z \notin \Pi \\[1mm] \Re z \geq \mu_n}} \, \sum_{k=N_0+1}^{\infty} 
\frac{1}{k^{\omega}|\mu_k -z|} 
\leq \sup_{j\geq n}\sup_{z_j \in \Xi_j} \, \sum_{k=N_0+1}^{\infty} 
\frac{1}{k^{\omega}|\mu_k -z_j|} 
\\
& \qquad \leq \sup_{ j \geq n} 
\left(
\sum_{k=N_0+1}^{j-3} \frac{1}{k^\omega (\mu_{j-1} -\mu_k)} 
+  \sum_{k=j+3}^{\infty} \frac{1}{k^\omega (\mu_{k} -\mu_{j+1})}
\right)
\\& \qquad \quad 
+ \sup_{j\geq n} \sup_{z_j \in \Xi_j} \sum_{k=j-2}^{j+3} \frac{1}{k^\omega |\mu_k -z_j|}.
\end{aligned}
\end{equation}
From the definition of sets $\Xi_j$, see \eqref{Xi.j.def}, the last term on the right of \eqref{sum.inf.est} is $\BigO(n^{1-\omega-\gamma})$ as $n \to \infty$. The remaining terms in \eqref{sum.inf.est} are estimated using Lemma~\ref{lem:sum.dist} in the first step and Lemma~\ref{lem:sum.est.2} in the second step
\begin{equation}\label{sum.inf.est.f}
\begin{aligned}
&\sup_{j\geq n} \left(
\sum_{k=N_0+1}^{j-3} \frac{1}{k^\omega (\mu_{j-1} -\mu_k)} 
+  \sum_{k=j+3}^{\infty} \frac{1}{k^\omega (\mu_{k} -\mu_{j+1})}\right)
\\
& \ \ \leq 
\frac{\kappa}{\gamma} \sup_{j\geq n}  
\left(
\sum_{k=N_0+1}^{j-3} \frac{1}{k^\omega \left((j-2)^\gamma - k^\gamma \right)} 
+ \sum_{k=j+3}^{\infty} \frac{1}{k^\omega ((k-1)^\gamma -(j+1)^\gamma)}
\right)
\\
& \quad = \BigO(\sigma_{\omega,\gamma}(n)), \quad n \to \infty.
\end{aligned}
\end{equation}
Putting \eqref{sum.ag.n.split}--\eqref{sum.inf.est.f} together, we get \eqref{sum.ag.n}.
\end{proof}

\section{Simple examples}
\label{sec:ex.sim}

We analyze perturbations of several simple operators. First in Section \ref{subsec:opt}, following the constructions in \cite[Sec.6.3]{Adduci-2012-10} and \cite[Sec.8.1]{Adduci-2012-73}, we provide examples of self-adjoint operators $A$ and perturbations $B$ showing that the condition $2\alpha + \gamma> 1$ cannot be weakened to $2\alpha + \gamma= 1$.
Next, we consider perturbations of finite band infinite matrices and $-\dd^2/\dd x^2$ on a finite interval with Neumann boundary conditions, see Sections \ref{subsec:finite.band}, \ref{subsec:Neumann}, respectively. In both subsections, we give conditions on the self-adjoint operator and the perturbations guaranteeing that the assumptions \eqref{asm:A} and \eqref{asm:b} hold and thus the results of Section~\ref{sec:abstract} are applicable. In particular, the eigensystem of the perturbed operator $T$ contains a Riesz basis. 

Some of the conclusions in Sections \ref{subsec:finite.band}, \ref{subsec:Neumann} are certainly not new, see \eg~\cite{Adduci-2010-91} for eigenvalue analysis of tridiagonal matrices, \cite[Chap.XIX.3]{DS} or \cite{Mikhajlov-1962-3} for Riesz basis property of perturbations of $-\dd^2/\dd x^2$ in boundary conditions or \cite{Djakov-2009-481,Djakov-2009-6,Djakov-2010-203} for perturbations by singular potentials, nevertheless, the goal of these sections is to demonstrate the flexibility of our approach. For instance, $-\dd^2/\dd x^2$ with an infinite number of complex $\delta$-interactions can be treated with the same amount of effort as the perturbation by a bounded potential.

\subsection{Optimality of the condition~\eqref{asm:b}}
\label{subsec:opt}

Consider $\H = \ell^2(\N)$, its standard basis $\{e_k\}$ and define $A e_k := \mu_k e_k$, $k\in \N$, where $\mu_k =k^\gamma$, $\gamma >0$. We denote by 
\begin{equation}
d_k:=\frac 12 (\mu_{2k}-\mu_{2k-1}) = 2^{\gamma -2} k^{\gamma-1} (\gamma + \BigO(k^{-1})) , \quad k \to \infty,
\end{equation}
and define the perturbation
\begin{equation}
B e_{2k-1} := -d_k t_k e_{2k}, \quad  B e_{2k} := d_k t_k e_{2k-1}, \quad k \in \N,
\end{equation}
where $\{t_k\} \subset (0,1)$ and $t_k \to 1$ as $k \to \infty$. 

While the condition \eqref{asm:A} is clearly satisfied, regarding \eqref{asm:b} we get
\begin{equation}
|\langle B e_m, e_n \rangle| \leq \min \{ \|Be_m\|, \|Be_n\|\} \leq \|Be_m\|^\frac 12 \|Be_n\|^\frac 12 \leq C (mn)^\frac{\gamma-1}{2},
\end{equation}
so $2\alpha + \gamma= 1$. We show below by elementary explicit calculations that the eigensystem of $T:=A+B$ does not contain even a basis.

Since the perturbation $B$ is block-diagonal, it suffices to analyze the 2-dimensional blocks corresponding to $\lspan\{e_{2k-1},e_{2k}\}$,
\begin{equation}
T_k:=A_k + B_k =
\begin{pmatrix}
\mu_{2k-1} + d_k & 0 \\
0 & \mu_{2k-1} +d_k
\end{pmatrix}
+
d_k
\begin{pmatrix}
-1  & t_k \\
-t_k & 1
\end{pmatrix}.
\end{equation}
Eigenvalues of $T_k$ can be of course calculated explicitly, see \cite[Sec.6.3]{Adduci-2012-10} and \cite[Sec.8.1]{Adduci-2012-73} for details, namely
\begin{equation}
\begin{aligned}
T_k g_k^\pm &= (\mu_{2k-1} + d_k \pm d_k \tau_k) g_k^\pm, 
\\ \tau_k &= \sqrt{1-t_k^2}, \quad 
g_k^\pm = \begin{pmatrix}
1 \\ G_k^\pm
\end{pmatrix},
\quad 
G_k= \left(\frac{1+\tau_k}{1-\tau_k} \right)^\frac 12.
\end{aligned}
\end{equation}
The norms of spectral projections $P_k^\pm$ of $T$ related to eigenvalues $\mu_{2k-1} + d_k \pm d_k \tau_k$ are explicit as well. We denote by $\{(g_k^\pm)^*\}$ the biorthonormal vectors to $\{g_k^\pm\}$, \ie~$\langle g_k^\nu, (g_k^\mu)^* \rangle = \delta_{\nu,\mu}$, then
\begin{equation}
\|P_k^\pm\| = \| \langle \cdot, (g_k^\pm)^* \rangle g_k^\pm \| = \|(g_k^\pm)^*\| \|g_k^\pm \| = \frac{1}{\tau_k^2} = \frac{1}{1 - t_k^2}.
\end{equation}
Hence for $\{t_k\}$, $t_k \to 1$, the eigensystem of $T$ does not contain a basis. 

\subsection{Finite band infinite matrices}
\label{subsec:finite.band}
Let $\cH=\ell^2(\N)$, $\gamma >0$ and 
\begin{equation}
A =
\begin{pmatrix}
a_1 & 0 & 0 & 0 & . 
\\
0 & a_2 & 0 & 0 & . 
\\
0 & 0 & a_3 & 0 & . 
\\
0 & 0 & 0 & a_4 & . 
\\
. & . & . & . & . 
\end{pmatrix},
\qquad
B =
\begin{pmatrix}
b_1^{(0)} & b_1^{(1)} & 0 & 0 & . 
\\
b_1^{(-1)} & b_2^{(0)} & b_2^{(1)} & 0 & . 
\\
0 & b_2^{(-1)} & b_3^{(0)} & b_3^{(1)} & . 
\\
0 & 0 & b_3^{(-1)} & b_4^{(0)} & . 
\\
. & . & . & . & . 
\end{pmatrix},
\end{equation}
where, with some $M>0$, 
\begin{equation}
a_k = k^\gamma, \quad |b_k^{(j)}| \leq M k^\omega, \quad j\in \{-1,0,1\}, \ k \in \N.
\end{equation}
Clearly, $A$ considered with the maximal domain satisfies assumption \eqref{asm:A} with $\psi_k = e_k$, $k \in \N$, where $\{e_k\}$ is the standard basis of $\ell^2(\N)$, and $\mu_k = k^\gamma$, $k \in \N$.
The form $b=\langle B\cdot, \cdot \rangle$, generated by $B$, satisfies 
\begin{equation}
|b(\psi_m, \psi_n)| \leq C \min \{m^\omega, n^\omega\} \leq C m^{\frac\omega 2} n^{\frac\omega 2},
\end{equation}
where $C>0$ is independent of $m,n$. 
Thus, the condition \eqref{asm:b} is satisfied if 
\begin{equation}\label{fin.band.cond}
\omega < \gamma -1.
\end{equation}
The tri-diagonal perturbation $B$ can be replaced by a finite band matrix with off-diagonal sequences $\{b_k^{(j)}\}_{k=1}^\infty$, $j \in \{-j_0,\dots,j_0\}$, satisfying
\begin{equation}
|b_k^{(j)}| \leq M k^\omega, \quad j\in \{-j_0,\dots,j_0\}, \ k \in \N
\end{equation}
and it is easy to see that the condition~\eqref{asm:b} is satisfied if \eqref{fin.band.cond} holds.

\subsection{Perturbations of  $-\dd^2/\dd x^2$ on a finite interval with Neumann boundary conditions.}
\label{subsec:Neumann}

Let $l \in (0,\infty)$ and consider the self-adjoint operator $A$ and the associated form $a$
\begin{equation} \label{NL.def}
\begin{aligned}
A &= - \frac{\dd^2}{\dd x^2}, & \Dom(A) & = \{ \psi \in W^{2,2}(-l,l): \psi'(\pm l)=0\}, \\
a[\psi] & = \|\psi'\|^2, & \Dom(a) & = \{ \psi \in W^{1,2}(-l,l)\}.
\end{aligned}
\end{equation}
Eigenvalues of $A$ and related orthonormal eigenfunctions read
\begin{equation}\label{psi.cos}
\mu_k = \left(\frac{k \pi}{2l} \right)^2, \quad k \in \N_0, \qquad \psi_k(x) = 
\begin{cases}\displaystyle
\frac{1}{\sqrt{2l}}, & k =0,\\[4mm]
\displaystyle
\frac 1{\sqrt{l}} \cos(\sqrt{\mu_k} (x+l)),&  k \in \N, 
\end{cases}
\end{equation}
thus the condition \eqref{asm:A} is satisfied with $\gamma=2$.
We analyze several perturbations of this $A$, in boundary conditions, by $\delta$-interactions, by $L^1$ and singular potentials. Results and proofs for perturbations of $-\dd^2/\dd x^2$ with Dirichlet boundary are completely analogous.

\subsubsection{Robin boundary conditions}
Consider the form 
\begin{equation}
\begin{aligned}
b_{\rm R}[\psi] &:= \nu_+ |\psi(l)|^2 - \nu_- |\psi(-l)|^2, \quad \nu_\pm \in \C, \quad \psi \in \Dom(a).
%\\
%\Dom(b_{\rm R}) & := \{ \psi \in W^{1,2}(0,\pi)\};
\end{aligned}
\end{equation}
The m-sectorial operator $T_{\rm R}$ associated with the form $t:=a+b_{\rm R}$ is actually $-\dd^2/\dd x^2$ with Robin boundary conditions
\begin{equation} \label{TR.def}
T_{\rm R}  = -\frac{\dd^2}{\dd x^2}, \qquad
\Dom(T_{\rm R}) = \{ \psi \in W^{2,2}(-l,l): \psi'(\pm l)+ \nu_\pm \psi(\pm l)=0, \},
\end{equation}
see \eg~\cite[Ex.VI.2.16]{Kato-1966}.
Since $\{\psi_m\}$ from \eqref{psi.cos} are uniformly bounded, we have $\sup_{m,n \in \N_0}|b_{\rm R}(\psi_m,\psi_n)| < \infty$, thus $b_{\rm R}$ satisfies the condition \eqref{asm:b} with $\alpha=0$.

\subsubsection{$\delta$-interactions}
\label{subsubsec:Lapl.delta}
The $\delta$-potential placed at $x_0$ with a complex coupling $\nu$ generates the form 
\begin{equation}
b_\delta[\psi] := \nu |\psi(x_0)|^2, \quad \nu \in \C, \quad x_0 \in (-l,l), \quad \psi \in \Dom(a).
\end{equation}
It satisfies the condition \eqref{asm:b} with $\alpha=0$ since $|b_\delta(\psi_m,\psi_n)| \leq |\nu|/l$, see \eqref{psi.cos}. In fact, the condition \eqref{asm:b} is satisfied with $\alpha=0$ also for an infinite number of $\delta$'s
\begin{equation}
\sum_{k=1}^\infty \nu_k \delta(x-x_k), \quad \{\nu_k\} \in \ell^1(\N), \quad \{x_k\} \subset (-l,l). 
\end{equation}
Indeed, the corresponding form reads
\begin{equation}
b_\delta^\infty[\psi] := \sum_{k=1}^\infty \nu_k |\psi(x_k)|^2, \quad \{\nu_k\} \in \ell^1(\N), \quad \{x_k\} \subset (-l,l), \quad \psi \in \Dom(a)
\end{equation}
and we have $|b_\delta^\infty(\psi_m,\psi_n)| \leq \|\nu\|_{\ell^1(\N)}/l$, see \eqref{psi.cos}.

\subsubsection{$L^1$-potential} A function $V \in L^1(-l,l)$ generates the form
\begin{equation}\label{V.form.int}
b_V[\psi]:= \int_{-l}^l V(x) \, |\psi(x)|^2 \, \dd x, \quad  \psi \in \Dom(a).
\end{equation}
Since $|b_V(\psi_m,\psi_n)| \leq \|V\|_{L^1(-l,l)}/l$, see~\eqref{psi.cos}, 
the condition \eqref{asm:b} is satisfied with $\alpha =0$. Notice that the classical formula is recovered by using Riemann-Lebesgue lemma, namely, the first correction $\lambda_n^{(1)}$ to eigenvalues of $A$, see Theorem~\ref{thm:ev.asym}, reads
\begin{equation}\label{SL.la.asym}
\begin{aligned}
\lambda_n^{(1)} &= \frac 1 l \int_{-l}^l V(x) \cos^2 (\sqrt{\mu_n}(x+l)) \, \dd x 
\\ &= \frac 1 {2l} \int_{-l}^l V(x) \, \dd x 
+ \frac{1}{2l} \int_{-l}^l V(x) \cos (2\sqrt{\mu_n}(x+l)) \, \dd x 
\\
&= \frac 1 {2l} \int_{-l}^l V(x) \, \dd x  + o(1), \quad n \to \infty.
\end{aligned}
\end{equation}

\subsubsection{Singular potentials} 
\label{subsubsec:Lapl.sing}
Consider $V \in W^{-s,2}(-l,l)$ with some $ s \geq 0$, so 
\begin{equation}\label{V.sing.N}
\exists C>0, \ \exists s \geq 0, \ \forall \phi \in W^{1,2}(-l,l), \ |(V,\phi)| \leq C \|\phi\|_{W^{1,2}(-l,l)}^s \|\phi\|^{1-s}.
\end{equation}
The associated form, a generalization of \eqref{V.form.int}, reads
\begin{equation}
b_V(\phi,\psi):=(V,\phi\overline{\psi}), \quad \phi, \psi \in \Dom(a).
\end{equation}
For $\{\psi_m\}$ from \eqref{psi.cos}, we have $\|\psi_m\| = \BigO(1)$  and $\|\psi_m'\| = \BigO(m)$ as $m \to \infty$, thus we get from \eqref{V.sing.N} that
\begin{equation}
\begin{aligned}
|b_V(\psi_m,\psi_n)| &= |(V,\psi_m \psi_n)| 
\leq 
C (\|(\psi_m \psi_n)'\|^2 +  \|\psi_m \psi_n\|^2)^\frac s2 \|\psi_m \psi_n\|^{1-s}
\\
& 
\leq 
C_1 (m+n)^s 
\leq
C_2 (mn)^s.
\end{aligned}
\end{equation}
Hence $b_V$ satisfies the condition \eqref{asm:b} if $s < 1/2$.

\section{Perturbations of single-well Schr\"odinger operators} 
\label{sec:aho}

Our main examples are perturbations of self-adjoint Schr\"odinger operators $A$ in $L^2(\R)$ with the associated quadratic forms $a$
\begin{equation} \label{aho.def}
\begin{aligned}
A &= -\frac{\dd^2}{\dd x^2} + Q(x),&
\Dom(A)&  = \{ \psi \in W^{2,2}(\R): Q \psi \in L^2(\R)\}, 
\\
a[\psi] &= \|\psi'\|^2 + \|Q^\frac 12 \psi\|^2, &
\Dom(a)& = \{\psi \in W^{1,2}(\R): Q^\frac12 \psi \in L^2(\R) \}.
\end{aligned}
\end{equation}
The real potential $Q$ is assumed to satisfy the following.
\setcounter{asmR}{16}
\begin{asmR}\label{asm:Q}
Suppose that $Q \in C^2(\R) \cap C^3(\R\setminus [-l,l])$ with some $l \geq 0$, $Q \geq 0$, $Q(0)=0$, $Q$ is even, increasing on $(0,\infty)$, $Q'$ is non-decreasing on $(0,\infty)$,  
\begin{equation}\label{Q.asym}
\frac{Q^{(k+1)}(x)}{Q^{(k)}(x)} = \BigO \left( \frac1x \right), \quad x \to + \infty, \qquad k =0,1,2,  
\end{equation}
and
\begin{equation}\label{Q.beta}
\exists \beta > 1, \quad \lim_{x \to \infty} \frac{Q(x)}{|x|^\beta} = 1.
\end{equation}
\end{asmR}

Under Assumption~\ref{asm:Q}, the Schr\"odinger operator $A$ in \eqref{aho.def} is 
self-adjoint, bounded from below and has a compact resolvent; see \eg~\cite[Thm.XII.67]{Reed4}.

Recall that \eqref{Q.asym} with $k=0$ and Gronwall's inequality imply that $Q$ cannot grow faster than a polynomial, nonetheless, we assume the precise behavior at infinity in \eqref{Q.beta}. Moreover, \eqref{Q.asym} implies further that, for every $\sigma \in (0,\infty)$ and $k\in\{0,1,2\}$
\begin{equation}\label{Q.sigma}
\frac{Q^{(k)}(\sigma x)}{Q^{(k)}(x)} = \BigO(1), \quad x \to \infty,
\end{equation}
see \cite[\S 22.27]{Titchmarsh-1958-book2}. For $k=0$, \eqref{Q.sigma} follows from (in the non-obvious case $\sigma>1$)
\begin{equation}\label{Q.sigma.log}
\log \frac{Q(\sigma x)}{Q(x)} = \int_x^{\sigma x} \frac{Q'(t)}{Q(t)} \, \dd t = \BigO \left(\int_x^{\sigma x} \frac{\dd t}{t}  \right) = \BigO(1), \quad x \to \infty;
\end{equation}
the other cases are similar. The additional condition \eqref{Q.beta} implies more; namely, for every $\sigma \in (0,\infty)$,
\begin{equation}\label{Q.sigma.beta}
\lim_{x \to +\infty}\frac{Q(\sigma x)}{Q(x)} = \sigma^\beta \lim_{x \to +\infty}  \frac{Q(\sigma x) }{(\sigma x)^\beta} \frac{x^\beta}{Q(x)} = \sigma^\beta.
\end{equation}
Recall also that since $Q'$ is non-decreasing on $(0,\infty)$, we have
\begin{equation}\label{Q.convex}
Q(y) - Q(x) \leq Q'(y) (y-x), \quad x,y>0.
\end{equation}

We define the (positive) turning points $x_\mu$ and the associated quantity $a_\mu$ for $Q'$ by relations 
\begin{equation}\label{xmu.def}
Q(x_\mu) = \mu,\quad  x_\mu > 0, \qquad a_\mu:=Q'(x_\mu).
\end{equation}
Notice that the assumption \eqref{Q.beta} implies
\begin{equation}\label{x.mu.beta}
\lim_{\mu \to +\infty}\frac{\mu^\frac 1\beta}{x_\mu} = \lim_{\mu \to +\infty} \left(\frac{Q(x_\mu)}{x_\mu^\beta} \right)^\frac 1\beta =  1.
\end{equation}

\subsection{Eigenvalues of $A$}

The spectrum of $A$ contains only simple discrete eigenvalues $\{\mu_k\}_{k \in \N} \subset \R_+$ which are known to obey, see~\cite[Sec.7.7]{Titchmarsh-1962-book1}, 
\begin{equation}\label{EV.Tit}
\int_{-x_{\mu_k}}^{x_{\mu_k}} (\mu_k - Q(x))^\frac 12\, \dd x = \left( k+\frac{1}{2}\right) \pi + \BigO(k^{-1}), \quad k \to \infty.
\end{equation}
The result \eqref{EV.Tit} holds also if \eqref{Q.beta} is replaced by
\begin{equation}\label{Q.unbd}
\lim_{x \to \infty} Q(x) = +\infty.
\end{equation}
More can be said under the additional condition \eqref{Q.beta}.

\begin{proposition}\label{prop:Q.gaps}
Let $Q$ satisfy Assumption~\ref{asm:Q}. Then the eigenvalues $\{\mu_k\}$ of the operator $A$ in \eqref{aho.def} satisfy
\begin{align}
\mu_k &= 
\left( \frac{\pi}{\Omega_\beta} k \right)^\gamma(1+o(1)), & k \to \infty, \label{Q.beta.EV.asym}
\\
\mu_{k+1} - \mu_k & = \frac{2\pi}{\Omega_\beta'} 
\left(
\frac{\pi}{\Omega_\beta} k
\right)^{\gamma-1}   (1+o(1)), & k \to \infty, \label{Q.beta.EV.gaps}
\end{align}
where 
\begin{equation}\label{Omega.beta}
\gamma = \frac{2 \beta}{\beta+2}, \qquad \Omega_\beta= 2 \int_0^1(1-t^\beta)^\frac 12 \, \dd t, \qquad  \Omega_\beta' = 2 \int_0^1 \frac{\dd t}{(1-t^\beta)^\frac12}.
\end{equation}
\end{proposition}
Before we give the proof of Proposition~\ref{prop:Q.gaps} we show two following.

\begin{lemma}
Let $Q$ satisfy Assumption~\ref{asm:Q}. Then
\begin{equation}\label{Omega.lim}
\begin{aligned}
\lim_{x\to +\infty} 2 \int_0^1 \left( 1 - \frac{Q(x t)}{Q(x)} \right)^\frac12 \, \dd t
&=\Omega_\beta, 
\\
\lim_{x\to +\infty} 2 \int_0^1 \left( 1 - \frac{Q(x t)}{Q(x)} \right)^{-\frac 12}\, \dd t 
&= \Omega_\beta'.
\end{aligned}
\end{equation}
\end{lemma}
\begin{proof}
Since $Q$ is increasing, we have $Q(x t)/Q(x) \leq 1$ for all $t \in [0,1]$, thus the dominated convergence theorem and~\eqref{Q.sigma.beta} justifies the first limit in \eqref{Omega.lim}.

To show the second limit, we analyze separately the cases $0\leq t \leq 1/2$ and $1/2 \leq t \leq 1$.
For the former, we have from \eqref{Q.sigma.beta} that
\begin{equation}
\left(1 - \frac{Q(x t)}{Q(x)}\right)^{-\frac 12} 
\leq 
\left(1 - \frac{Q(\frac x2)}{Q(x)}\right)^{-\frac 12} 
\to 
\left(1 - 2^{-\beta}\right)^{-\frac 12}; 
\end{equation}
for the latter, we get from the mean value theorem, \eqref{Q.convex} with $Q(0)=0$ and \eqref{Q.sigma.beta} that
\begin{equation}
\begin{aligned}
\left(1 - \frac{Q(x t)}{Q(x)}\right)^{-\frac 12} 
&\leq 
\left(\frac{Q(x)}{Q'(x t) x (1-t)}\right)^{\frac 12} 
\leq 
\left(\frac{Q(x)}{Q'(\frac x2) x (1-t)}\right)^{\frac 12} 
\\
&\leq 
\left(\frac{Q(x)}{2 Q(\frac x2) (1-t)}\right)^{\frac 12} \to \frac{2^\frac{\beta-1}{2}}{(1-t)^\frac{1}{2}}.
\end{aligned}
\end{equation}
Hence there is a constant $C>0$ such that for all sufficiently large $x$ and all $0 < t <1$ 
\begin{equation}
\left(1 - \frac{Q(x t)}{Q(x)}\right)^{-\frac 12} 
\leq \frac{C}{(1-t)^\frac 12},
\end{equation}
which is integrable on $(0,1)$ and the second limit in \eqref{Omega.lim} is justified by the dominated convergence theorem and \eqref{Q.sigma.beta}.
\end{proof}

\begin{proof}[Proof of Proposition~\ref{prop:Q.gaps}]
Simple manipulations with \eqref{EV.Tit} leads to 
\begin{equation}\label{EV.Q.beta.1}
2 \mu_k^\frac 12 x_{\mu_k} \int_0^1 \left( 1 - \frac{Q(x_{\mu_k} t)}{Q(x_{\mu_k})} \right)^\frac12 \, \dd t = \pi k (1+o(1)), \quad k \to \infty.
\end{equation}
Hence \eqref{Q.beta.EV.asym} follows from \eqref{EV.Q.beta.1}, \eqref{x.mu.beta} and \eqref{Omega.lim}.

To show \eqref{Q.beta.EV.gaps}, we define the function 
\begin{equation}
g(\mu):= \int_0^{x_\mu} (\mu - Q(x))^\frac 12\, \dd x
\end{equation}
and start with the identity obtained from \eqref{EV.Tit}
\begin{equation}\label{g.muk}
g(\mu_{k+1}) - g(\mu_k) = \frac{\pi}{2}(1 + o(1)), \quad k \to \infty.
\end{equation}
Observing that
\begin{equation}
\lim_{x \to x_\mu}\frac{x_\mu - x}{\mu - Q(x)} = \frac{1}{a_\mu},
\end{equation}
we can check that $g$ is differentiable and, after a change of variables,  
\begin{equation}\label{g'mu}
g'(\mu) = \frac {x_\mu}{2 \mu^\frac12}\int_0^1 \left( 1 - \frac{Q(x t)}{Q(x)} \right)^{-\frac 12} \; \dd t.
\end{equation}
%
%Moreover, we show below that 
%%
%\begin{equation}\label{Omega'}
%\lim_{\mu \to \infty} \frac{4 g'(\mu)}{\mu^\frac{2-\beta}{2\beta}} = \Omega_{\beta}'.
%\end{equation}
%
The mean value theorem yields (with $\eta_k \in (\mu_k,\mu_{k+1})$)
\begin{equation}
g(\mu_{k+1}) - g(\mu_k) = g'(\eta_k)(\mu_{k+1} - \mu_k),
\end{equation}
and therefore, using~\eqref{g.muk}, \eqref{g'mu}, \eqref{Omega.lim}, \eqref{Q.sigma.beta} and \eqref{x.mu.beta}, we obtain
\begin{equation}
\mu_{k+1} - \mu_k 
= \frac{\pi}{2 g'(\eta_k)}(1+o(1)) 
= \frac{2\pi}{\Omega_{\beta}'} \mu_k^\frac{\beta-2}{2\beta}(1+o(1)).
\end{equation}
Then \eqref{Q.beta.EV.gaps} follows by employing \eqref{Q.beta.EV.asym}.
\end{proof}

\subsection{Eigenfunctions of $A$ and their weighted $L^q$-norms}
Since $Q$ is even, orthonormal eigenfunctions $\{\psi_k\}$, related to eigenvalues $\{\mu_k\}$, are even or odd functions. Moreover, even with \eqref{Q.beta} replaced by \eqref{Q.unbd}, $\{\psi_k\}$ can be chosen such that they satisfy (see \eg~\cite[\S 22.27]{Titchmarsh-1958-book2} and \cite{Giertz-1964-14})  
\begin{equation}\label{psi.u.rel}
\psi_k(x) = \frac{1}{\|u_k\|} u_k(x) (1+ \BigO(x_{\mu_k}^{-1} \mu_k^{-\frac12})), \qquad    x >0,
\end{equation}
where $u_k= u(x, \mu_k)$ with
\begin{align}
u & = u(x,\mu) = \left(\frac{\zeta}{\zeta'} \right)^{\frac12} K_{\frac13}(-i \zeta), 
\label{u.def}
\\
\zeta &= \zeta(x,\mu)= 
\begin{cases}
\displaystyle
\int_x^{x_\mu} (\mu - Q(s))^\frac12 \, \dd s, & 0<x<x_\mu, 
\label{zeta.def}
\\[4mm]
\displaystyle
i \int_{x_\mu}^x (Q(s)-\mu)^\frac12 \, \dd s, & x > x_\mu;
\end{cases}
\end{align}
$K_{\frac13}$ is the modified Bessel function of order $1/3$. Using the asymptotic formulas for Bessel functions, we get further that, see \eg~\cite{Giertz-1964-14},
\begin{equation}\label{u^2.exp}
u^2(x) = \frac{\pi}{(\mu-Q(x))^\frac12} (1 + \sin 2 \zeta + R_1(\zeta)), \quad \zeta >1,
\end{equation}
where $|R_1(\zeta)|<1/(2\zeta)$, and 
\begin{equation}\label{|u|.est}
|u(x)| \leq 
\begin{cases}
\displaystyle \frac{A_1}{(\mu - Q(x))^{\frac14}}, & 0 \leq x < x_\mu - \delta, \\[4mm]
\displaystyle \frac{A_1}{(\mu - Q(x_\mu - \delta))^{\frac14}}, & x_\mu - \delta \leq x \leq x_\mu + \delta_1, \\[4mm]
\displaystyle \frac{A_1}{2(Q(x)-\mu)^{\frac14}} 
e^{	- \int_{x_\mu}^x (Q(s) - \mu)^\frac12 \, \dd s}  
, 
& x > x_\mu + \delta_1, 
\end{cases}
\end{equation}
where $A_1 =2.7$ and numbers $\delta$, $\delta_1$ are defined by equations
\begin{equation}\label{delta.def}
\zeta(x_\mu-\delta) = - i \zeta(x_\mu + \delta_1) = 1.
\end{equation} 
It can be shown, see \eg~the appendix of \cite{Mityagin-2013arx}, that
\begin{equation}\label{delta.amu}
\delta, \delta_1 = \BigO(a_\mu^{-\frac13}), \qquad \delta^{-1}, \delta_1^{-1} = \BigO(a_\mu^{\frac13}). 
\end{equation}
Further, it is known, see \cite[Lemma~5]{Giertz-1964-14}, that 
\begin{equation}\label{u.norm}
\int_0^{\infty} u^2(x) \, \dd x 
= 
\int_0^{x_\mu} \frac{\pi \, \dd x}{(\mu - Q(x))^\frac12} 
\left(
1 + \BigO(x_\mu^{-\frac13} \mu^{-\frac16} )
\right). 
\end{equation}
Under the assumption \eqref{Q.beta} we therefore obtain from \eqref{Omega.lim} and \eqref{x.mu.beta} that
\begin{equation}\label{u.norm.beta}
\|u_k\|^2 = \pi \Omega_\beta' \mu_k^\frac{2-\beta}{2\beta} (1+o(1))
= \pi \Omega_\beta' 
\left(
\frac{\pi}{\Omega_\beta} k
\right)^\frac{2-\beta}{2+\beta}
(1+o(1)), \quad k \to \infty,  
\end{equation}
where $\Omega_\beta$, $\Omega_\beta'$ are as in \eqref{Omega.beta}.

Finally, we recall the pointwise estimates for $\{\psi_k\}$, see \cite[Sec.8]{Titchmarsh-1962-book1}, 
\begin{align}
|\psi_k(x)| &= \BigO ( x_{\mu_k}^{-\frac12} ) = \BigO \left( k^{-\frac{1}{\beta+2}}\right), \label{|psi|}
\\
|\psi_k'(x)| &= \BigO ( \mu_k^\frac12 x_{\mu_k}^{-\frac12} ) = \BigO \left( k^{\frac{\beta-1}{\beta+2}}\right),
\label{|psi'|}
\\
\|\psi_k\|_{L^\infty(\R)} &= \BigO ( \mu_k^\frac14 x_{\mu_k}^{-\frac12} a_{\mu_k}^{-\frac16} )
= \BigO \left( k^\frac{\beta-4}{6(\beta+2)} \right) , \quad k \to \infty;
\label{sup|psi|}
\end{align}
the first equalities hold also if \eqref{Q.beta} is replaced by \eqref{Q.unbd}; notice that the point $x \in \R$ is arbitrary but fixed in \eqref{|psi|} and \eqref{|psi'|}.

Next, we estimate the weighted $L^q$-norms of $\{\psi_k\}$. For $\beta=2$ and without the weight, we recover the known results for Hermite functions, see \cite[Lemma~1.5.2]{Thangavelu-1993-42}, where in fact both-sided estimates are given. For $q=2$, and $Q$,~$w$~satisfying similar conditions like here, both-sided estimates (in fact limits) are established in \cite{Mityagin-2013arx}.
\begin{proposition}\label{prop:Q.Lq}
Let $Q$ satisfy Assumption \ref{asm:Q} with \eqref{Q.beta} replaced by \eqref{Q.unbd}, let $x_\mu$, $a_\mu$ be as in \eqref{xmu.def} and let $\{\psi_k\}$ be as in \eqref{psi.u.rel}. Suppose that the weight $w \in C^1(\R)$ is positive, even, increasing on $(0,\infty)$ and satisfy
\begin{equation}\label{w.1x}
\frac{w'(x)}{w(x)} = \BigO \left( \frac 1x \right), \quad x \to  \infty.
\end{equation}
Then 
\begin{equation}\label{psik.w.q}
\|w \, \psi_k\|_{L^q(\R)} = \BigO \left( w(x_{\mu_k}) \|\psi_k\|_{L^q(\R)} \right), \quad k \to \infty.
\end{equation}
Furthermore, as $k \to \infty$, 
\begin{equation}\label{psik.Lq}
\|\psi_k\|_{L^q(\R)}
=
\begin{cases}
\BigO \left(\left(a_{\mu_k} \mu_k^{-1}\right)^{\frac{q-2}{2q}} \right), & 1 \leq q < 4, 
\\[3mm] 
\BigO \left(\left(a_{\mu_k} \mu_k^{-1}\right)^{-\frac14} (\log (\mu_k a_{\mu_k}^{-\frac 23}))^\frac 14 \right), & q = 4, 
\\[3mm]
\BigO \left( a_{\mu_k}^{\frac{q-1}{3q}} \mu_k^{-\frac14} \right), & q > 4. 
\end{cases}
\end{equation}
If \eqref{Q.beta} is satisfied in addition, then, as $k \to \infty$,
\begin{equation}\label{aho.ef.est}
\|\psi_k\|_{L^{ q}(\R)}
= 
\begin{cases}
\BigO \left(k^\frac{2-q }{ q(\beta+2)} \right), & 1 \leq  q < 4, 
\\[2mm] 
\BigO \left(k^{-\frac{1}{2(\beta+2)}} (\log k)^\frac 14 \right), &  q = 4, 
\\[2mm]
\BigO \left(k^\frac{4 - 4\beta - 4  q +  q\beta}{6  q(\beta+2)} \right), &  q > 4. 
\end{cases}
\end{equation}
\end{proposition}
\begin{proof}
We suppress the subscript $k$ in the sequel and keep $\mu$ only.
The letter $C$ denotes a constant, which can vary in every step, however, it is independent of $\mu $. 
The case $q=\infty$ is reduces to \eqref{sup|psi|}, so we analyze $1\leq q < \infty$ only. 

Since $Q'$ is non-decreasing on $(0,\infty)$, we get from \eqref{u.norm} that
\begin{equation}\label{u.norm.low}
\int_0^{\infty} u^2 \, \dd x 
\geq C
\int_0^{x_\mu} \frac{\dd x}{(\mu - Q)^\frac12}
\geq 
\frac{1}{a_\mu}\int_0^{x_\mu} \frac{Q' \, \dd x}{(\mu - Q)^\frac12} 
= 
\frac{ \mu^\frac12}{a_\mu}.
\end{equation}
Thus, we have from \eqref{psi.u.rel} and \eqref{u.norm.low} that
\begin{equation}
\|w \, \psi\|_{L^q(\R)} 
 \leq  
C \|u\|^{-1} \|w\, u\|_{L^q(\R)}
\leq 
C a_\mu^\frac12 \mu^{-\frac14} \|w \; u\|_{L^q(\R)}.
\end{equation}

In the following, we split the integral and employ \eqref{|u|.est} in estimates,
\begin{equation}
\int_0^\infty |w u |^q \, \dd x 
= 
\left(
\int_0^{x_{\frac\mu2}} 
+ \int_{x_{\frac\mu2}}^{x_\mu - \delta} 
+ \int_{x_\mu - \delta}^{x_\mu + \delta_1} 
+ \int_{x_\mu + \delta_1}^{x_{\frac32 \mu}}  
+ \int_{x_{\frac32 \mu}}^\infty 
\right)
|w u|^q \, \dd x.
\end{equation}
%
%\begin{itemize}
$\bullet$ $x< x_{\frac\mu2}$: For all sufficiently large $\mu$, we have $x_{\frac\mu2} \leq x_\mu- \delta$. To see this, we use that $Q$ is increasing, $Q(x_{\mu/2})=\mu/2$ and by the mean value theorem and $Q'(x)/Q(x)=\BigO(1/x)$ we get
\begin{equation}
\frac{Q(x_\mu-\delta)}{Q(x_\mu)} = 1 - \frac{Q(x_\mu)-Q(x_\mu-\delta)}{Q(x_\mu)} = 1 + \BigO\left(\frac{\delta}{x_\mu}\right), \quad \mu \to \infty.
\end{equation}
Hence, using \eqref{|u|.est} and $Q'(x)/Q(x) = \BigO(1/x)$ in the last step, we obtain
\begin{equation}
\begin{aligned}
\int_0^{x_\frac\mu 2} 
|w u|^q \, \dd x
\leq 
 \frac{ C w(x_\mu)^q x_\mu}{(\mu - Q(x_{\frac\mu 2}))^{\frac q4}}
\leq 
C  w(x_\mu)^q x_\mu \mu^{-\frac q4}	
\leq C w(x_\mu)^q a_\mu^{-1} \mu^{1-\frac q4}.
\end{aligned}
\end{equation}
$\bullet$ $x_\frac \mu 2 < x < x_\mu - \delta$:
since $Q'$ is non-decreasing,
\begin{equation}
\int_{x_\frac\mu 2}^{x_\mu - \delta}
|w u|^q \, \dd x
\leq
C w(x_\mu)^q a_{\frac \mu2}^{-1}
\int_{x_\frac\mu 2}^{x_\mu - \delta} \frac{Q' \, \dd x}{(\mu - Q)^\frac q4}.
\end{equation}
We can replace $a_{\frac \mu 2}$ by $a_\mu$ since from \eqref{Q.convex} and $Q'(x)/Q(x) = \BigO(1/x)$, we get
\begin{equation}
1 \geq \frac{a_{\frac\mu 2}}{a_\mu} 
\geq 
\frac{Q(x_{\frac \mu 2})}{x_{\frac \mu 2} Q'(x_\mu) }
\geq C
\frac{Q(x_{\frac \mu 2}) x_\mu}{ x_{\frac \mu 2} Q(x_\mu) }
\geq 
\frac{C}{2}.
\end{equation}
Further,

\noindent
$1 \leq q <4$:
\begin{equation}
\int_{x_\frac\mu 2}^{x_\mu - \delta} \frac{Q' \, \dd x}{(\mu - Q)^\frac q4}
\leq
\frac{4-q}{4} (\mu - Q(x_\frac \mu2))^{1- \frac q4}
\leq 
C \mu^{1- \frac q4}.
\end{equation}
$q = 4$: 
by \eqref{Q.sigma} and \eqref{delta.amu},
\begin{equation}
\int_{x_\frac\mu 2}^{x_\mu - \delta} \frac{Q' \, \dd x}{\mu - Q}
 = 
\log \frac{\mu}{2 (Q(x_\mu) - Q(x_\mu-\delta))} 
\leq 
C  \log \frac{\mu}{a_\mu \delta}
\leq 
C \log (\mu a_\mu^{-\frac23}).
\end{equation}
$q>4:$ again by \eqref{Q.sigma} and \eqref{delta.amu},
\begin{equation}
\begin{aligned}
\int_{x_\frac\mu 2}^{x_\mu - \delta} \frac{Q' \, \dd x}{(\mu - Q)^\frac q4}
 \leq
\frac{q-4}{4(\mu - Q(x_\mu - \delta))^{\frac q4-1}} 
\leq 
\frac{C}{(a_\mu \delta)^{\frac q4-1}} 
%\\
 \leq
C a_\mu^{\frac 23- \frac{q}{6}}.
\end{aligned}
\end{equation}

In summary,
\begin{equation}
\int_{x_\frac\mu 2}^{x_\mu - \delta} |w u|^q \, \dd x
\leq
C w(x_\mu)^q a_\mu^{-1}
\begin{cases}
\mu^{1- \frac q4}, & 1 \leq q < 4, \\[1mm]  
\log (\mu a_\mu^{-\frac23}), &  q = 4, \\[1mm] 
a_\mu^{\frac 23- \frac{q}{6}}, & q>4. 
\end{cases}
\end{equation}
$\bullet$ $x_\mu - \delta < x < x_\mu + \delta_1$:
Notice that since $w$ satisfies \eqref{w.1x}, we have 
\begin{equation}\label{w.sigma}
\forall \sigma \in (0,\infty), \quad w(\sigma x) = \BigO(w(x)), \quad x \to +\infty,
\end{equation}
see \eqref{Q.sigma} and \eqref{Q.sigma.log}. Then by \eqref{|u|.est}, \eqref{Q.sigma} and \eqref{delta.amu} 
\begin{equation}
\int_{x_\mu - \delta}^{x_\mu + \delta_1} |w u|^q \, \dd x
\leq
C w(x_\mu)^q \frac{\delta + \delta_1}{Q'(x_\mu -\delta)^\frac q4 \delta^\frac q4}
\leq 
C w(x_\mu)^q a_\mu^{- \frac13 - \frac q6}.
\end{equation}
$\bullet$ $x_\mu + \delta_1 < x < x_{\frac32 \mu}$: like for $x_{\frac\mu2}$, it can showed that $x_{\frac32 \mu} \geq x_\mu + \delta_1$ for all sufficiently large $\mu$. Then, using \eqref{|u|.est}, we get
\begin{equation}
\int_{x_\mu + \delta_1}^{x_{\frac32 \mu}} |w u|^q \, \dd x
\leq 
C w(x_\mu)^q a_\mu^{-1} 
\int_{x_\mu + \delta_1}^{x_{\frac32 \mu}} \frac{Q' \, \dd x}{(Q-\mu)^\frac q4};
\end{equation}	
here $w(x_{\frac32 \mu})$ is replaced by $w(x_{\mu})$ since we have \eqref{w.sigma} and $x_{\frac32 \mu} = \BigO(x_\mu)$. To see the latter, we use that $Q'$ is non-decreasing on $(0,\infty)$ and \eqref{Q.convex} with $Q(0)=0$, 
\begin{equation}
\frac{x_{\frac32 \mu}}{x_\mu} 
= 
\frac{x_\mu + Q^{-1}(\frac32 \mu)-Q^{-1}(\mu)}{x_\mu} 
\leq  
\frac{x_\mu + \frac{\mu}{2 a_\mu}}{x_\mu} 
\leq \frac 32.  
\end{equation}	
Further,

\noindent
$1 \leq q <4$:
\begin{equation}
\int_{x_\mu +\delta_1}^{x_{\frac 32 \mu}} \frac{Q' \, \dd x}{(Q-\mu)^\frac q4}
\leq
\frac{4-q}{4} (Q(x_{\frac32 \mu})-\mu)^{1- \frac q4}
\leq 
C \mu^{1- \frac q4}.
\end{equation}
$q = 4$: by \eqref{delta.amu},
\begin{equation}
\int_{x_\mu + \delta_1}^{x_{\frac 32 \mu}} \frac{Q' \, \dd x}{Q-\mu}
=
\log \frac{\mu}{2 (Q(x_\mu + \delta_1) - Q(x_\mu))} 
\leq
C 
\log \frac{\mu}{a_\mu \delta_1}
\leq 
C \log (\mu a_\mu^{-\frac23}).
\end{equation}
$q>4:$ by \eqref{Q.sigma} and \eqref{delta.amu},
\begin{equation}
\int_{x_\mu +\delta_1}^{x_{\frac 32 \mu}} \frac{Q' \, \dd x}{(Q - \mu)^\frac q4}
\leq
\frac{q-4}{4} (Q(x_\mu + \delta_1) - \mu)^{1- \frac q4}
\leq 
C a_\mu^{\frac 23- \frac{q}{6}}.
\end{equation}

In summary,
\begin{equation}
\int_{x_\mu +\delta_1}^{x_{\frac 32 \mu}} |w u|^q \, \dd x
\leq
C w(x_\mu)^q a_\mu^{-1}
\begin{cases}
\mu^{1- \frac q4}, & 1 \leq q < 4, \\[1mm] 
\log (\mu a_\mu^{-\frac23}), &  q = 4, \\[1mm]
a_\mu^{\frac 23- \frac{q}{6}}, & q>4. 
\end{cases}
\end{equation}
$\bullet$ $ x_{\frac32 \mu} < x$: first, using $Q'(x)/Q(x) = \BigO(1/x)$, we get
\begin{equation}\label{zeta.32}
\int_{x_\mu}^{x} (Q - \mu)^\frac12 \, \dd s
\geq 
\frac 23 \frac{Q(x)^\frac32}{Q'(x)}   \left(1- \frac{\mu}{Q(x)} \right)^\frac 32 
\geq 
C \mu^\frac12 x. 
\end{equation}
Since $w$ does not grow faster than a polynomial, see \eqref{w.sigma} and Gronwall's inequality, we have from \eqref{|u|.est} that 
\begin{equation}\label{u.q.32}
\int_{x_{\frac32 \mu}}^\infty 
|w(x) u(x)|^q \, \dd x
\leq 
C \mu^{-\frac q4} \int_{x_{\frac32 \mu}}^\infty w(x) e^{-C \mu^\frac 12 x} \, \dd x
\leq
e^{-C \mu^\frac 12 x_\mu}.
\end{equation}

Putting all estimates from above together, we get
\begin{equation}
\|w \, \psi\|_{L^q(\R)}
\leq 
C w(x_{\mu})
\left(
a_{\mu}^{\frac 13 - \frac{1}{3q}} \mu^{-\frac14} + 
a_{\mu}^{\frac12 - \frac 1q} \mu^{-\frac 12 + \frac 1q} \iota_q(\mu)^\frac 1q
\right),
\end{equation}
where
\begin{equation}
\iota_q(\mu):=
\begin{cases}
1, & q \neq 4, 
\\ 
\log (\mu a_\mu^{-\frac 23}), & q = 4. 
\end{cases}
\end{equation}
Finally, for $1 \leq q <4$, 
\begin{equation}
a_{\mu}^{\frac 13 - \frac{1}{3q}} \mu^{-\frac14} + 
a_{\mu}^{\frac12 - \frac 1q} \mu^{-\frac 12 + \frac 1q}
= 
a_{\mu}^{\frac12 - \frac 1q} \mu^{-\frac 12 + \frac 1q}
\left(
1 + 
(a_\mu^{\frac 16} \mu^{-\frac 14})^\frac{4-q}{q}
\right)
\end{equation}
and, for $q >4$,
\begin{equation}
a_{\mu}^{\frac 13 - \frac{1}{3q}} \mu^{-\frac14} + 
a_{\mu}^{\frac12 - \frac 1q} \mu^{-\frac 12 + \frac 1q}
= 
a_{\mu}^{\frac 13 - \frac{1}{3q}} \mu^{-\frac14}  
\left(
1 + 
(a_\mu^{\frac 16} \mu^{-\frac 14} )^\frac{q-4}{q}
\right)
\end{equation}
thus \eqref{psik.w.q} and \eqref{psik.Lq} follow since $Q'(x)/Q(x) = \BigO(1/x)$ implies $a_\mu/\mu = \BigO(1/x_\mu)$.

If \eqref{Q.beta} holds in addition, we obtain \eqref{aho.ef.est} from \eqref{psik.Lq} by employing \eqref{x.mu.beta}, \eqref{Q.beta.EV.asym}, \eqref{Q.convex} and $Q'(x)/Q(x) = \BigO(1/x)$ as $x \to \infty$.
\end{proof}

\subsection{Perturbations by functional potentials $V$ in weighted $L^p$-spaces}

We define the following spaces
\begin{equation}\label{Lptau.def}
 L(p,\tau):= \left\{
 v:(1+x^2)^{-\frac\tau 2}|v(x)| \in L^p(\R) 
 \right\}, \quad 1 \leq  p \leq \infty,  \ \ \tau \in \R,
\end{equation}
as in \cite{Adduci-2012-10,Mityagin-2016-106}; notice that $L(p,0)=L^p(\R)$.

The form associated with the perturbation by a functional potential $V$ reads
\begin{equation}\label{bV.def}
b_V[\psi]:=\int_\R V |\psi|^2, \quad 
\Dom(b_V):=\{ \psi \in L^2(\R) \,: \, V |\psi|^2 \in L^1(\R) \}.
\end{equation}

\begin{theorem}\label{thm:aho.Lptau}
Let $Q$ satisfy Assumption~\ref{asm:Q} and let $A$ be the Schr\"odinger operator from~\eqref{aho.def}. 
Suppose that $V \in L(p,\tau)$ with $1 \leq p \leq \infty$, $\tau \geq 0$, and, depending on $p$, one of the following conditions is satisfied
\begin{equation}\label{p.tau.cond}
\begin{aligned}
\tau &<  \frac23 (\beta-1) \left(1 - \frac 1{2p} \right) &&
\quad \text{if} \quad 1 \leq p <2,
\\
\tau &< \frac{\beta-2}{2} + \frac 1p 
 &&
\quad \text{if} \quad 2 \leq  p \leq \infty.
\end{aligned}
\end{equation}
Then $A$ and the form $b_V$ from~\eqref{bV.def} satisfy conditions \eqref{asm:A} and \eqref{asm:b} with 
\begin{equation}
\gamma = \frac{2 \beta}{\beta+2}, 
\qquad
\alpha= \frac{1}{\beta+2}
\begin{cases}
\frac{\beta+2}{6} + \frac{1-\beta}{3p} - \tau, & 1 \leq p <2,
\\[2mm]
\frac 12 - \tau -\eps , & p = 2,
\\[2mm]
\frac1p - \tau, & p > 2,
\end{cases}
\end{equation}
where $\eps>0$ can be taken arbitrarily small.
\end{theorem}
\begin{proof}
It follows from Proposition~\ref{prop:Q.gaps} that the condition~\eqref{asm:A} is satisfied. To show that the condition~\eqref{asm:b} holds, we use the estimates for $L^q$-norms of $\{\psi_k\}$ from Proposition~\ref{prop:Q.Lq} with the weight $w(x)=(1+x^2)^{\tau/2}$. The rest is straightforward, like the proof of \cite[Thm.3]{Mityagin-2016-106}, 
\begin{equation}\label{bV.Lpt.est}
\begin{aligned}
|b_V(\psi_m,\psi_n)| & \leq \int_{\R} w^{-1} |V|  \, w |\psi_m| \, |\psi_n| \, \dd x 
\leq \|w^{-1} |V|\|_{L^p(\R)} \|w\psi_m \psi_n \|_{L^q(\R)} 
\\
& \leq C \|w^\frac12 \psi_m \|_{L^{2q}(\R)} \|w^\frac12 \psi_n \|_{L^{2q}(\R)},
\end{aligned}
\end{equation}
where $1/p +1/q =1$. The condition \eqref{asm:b} is satisfied due to \eqref{psik.w.q} and \eqref{aho.ef.est}.
\end{proof}

Putting together Theorems~\ref{thm:aho.Lptau} and \ref{thm:RB} we obtain the following claim on the eigensystem of the perturbed Schr\"odinger operators $T$. 

\begin{corollary}\label{cor:aho.Lptau}
Let $A$ be as in \eqref{aho.def} and $V \in L(p,\tau)$ with $p \in [1,\infty]$ and $\tau \geq 0$ satisfying \eqref{p.tau.cond}. Then the eigensystem of $T$, being the form sum of these $A$ and $V$, contains a Riesz basis.
\end{corollary}

\subsubsection{$L^1$-potentials with a controlled decay}

We consider a potential $V \in L^1(\R)$ with the decay $|x|^{-1-\eps}$ for some $\eps>0$ at infinity. More precisely, we suppose that
\begin{equation}\label{V.decay}
\begin{aligned}
&V= V_1 + V_2,
\\ &V_1 \in L^1(\R), \ \supp V_1 \ \text{is compact},
\\ &\exists \eps>0, \ V _2 \in L(\infty,-(1+\eps)).
\end{aligned}
\end{equation}
Since such a $V$ is integrable on $\R$, it follows from Theorem~\ref{thm:aho.Lptau} that the form associated with this $V$ satisfies the condition \eqref{asm:b} with $\alpha = (4-\beta)/(6(\beta+2))$. We show in the following that the latter improves if \eqref{V.decay} is satisfied, moreover, we derive a more convenient formula for the first correction $\lambda_n^{(1)}$ from Theorem~\ref{thm:ev.asym}.

\begin{theorem}\label{thm:aho.ev}
Let $Q$ satisfy Assumption~\ref{asm:Q}, let $A$ be the Schr\"odinger operator from~\eqref{aho.def} and let $V$ satisfy \eqref{V.decay}. Then the form $b_V$ from~\eqref{bV.def} satisfy the condition \eqref{asm:b} with 
\begin{equation}\label{alpha.dec}
\alpha = \frac{1}{\beta+2}.
\end{equation}
Moreover, the terms $\{\lambda_n^{(1)}\}$ from Theorem~\ref{thm:ev.asym} for $T$ being the form sum of $A$ and $V$ satisfy 
\begin{equation}\label{lam.1.dec}
\lambda_n^{(1)} =  
\frac{1}{\Omega_\beta'}
\left(
\frac{\pi}{\Omega_\beta} n
\right)^{-\frac{2}{\beta+2}} \int_\R V(x) \; \dd x + o \left(n^{-\frac{2}{\beta+2}}\right), \quad n \to \infty,
\end{equation}
where $\Omega_\beta$, $\Omega_\beta'$ are as in \eqref{Omega.beta}.

\end{theorem}
\begin{proof}
We show below that 
\begin{equation}\label{Vu^2.asym}
\int_0^\infty V  u^2  \, \dd x = \frac{\pi}{\mu^\frac12} \int_0^\infty V \; \dd x \; (1+o(1)), \quad \mu \to \infty.
\end{equation}
Hence, using that $u^2$ are even, \eqref{u.norm.beta}, \eqref{psi.u.rel} and
\begin{equation}
|b_V(\psi_m,\psi_n)| \leq \left(\int_\R |V||\psi_m|^2 \right)^\frac12 \left(\int_\R |V||\psi_n|^2 \right)^\frac12,
\end{equation}
we obtain that $b_V$ satisfies the condition \eqref{asm:b} with $\alpha$ in \eqref{alpha.dec}. The claim \eqref{lam.1.dec} follows from \eqref{EV.cor}, \eqref{Vu^2.asym}, \eqref{u.norm.beta} and \eqref{Q.beta.EV.asym}.

It remains to prove the key step \eqref{Vu^2.asym}. We analyze the integral in \eqref{Vu^2.asym} separately in several regions.

\noindent
$\bullet$ $0<x< x_{\sqrt \mu}$: As $\mu \to \infty$, we have
\begin{equation}\label{Q.mu12}
\frac{1}{(\mu-Q(x))^\frac12} - \frac{1}{\mu^\frac12} = \frac{Q(x)}{\mu^\frac12 (\mu-Q(x))^\frac12 \left(\mu^\frac12+(\mu-Q(x))^\frac12\right)} = \BigO(\mu^{-1}).
\end{equation}
Hence formula \eqref{u^2.exp} and $V \in L^1(\R)$ give
\begin{equation}
\int_0^{x_{\sqrt\mu}}V  u^2 \, \dd x  =
\frac{\pi}{\mu^\frac12} \int_0^{x_{\sqrt\mu}} V (1+\sin 2 \zeta + R_1(\zeta)) \, \dd x  + \BigO(\mu^{-1}), \quad \mu \to \infty.
\end{equation}
Next we show that 
\begin{equation}\label{int.V.sin}
\int_0^{x_{\sqrt\mu}}  V \sin 2 \zeta \, \dd x = o(1), \quad \mu \to \infty
\end{equation}
and
\begin{equation}\label{int.V.R}
\int_0^{x_{\sqrt\mu}}  V R_1(\zeta) \, \dd x = \BigO(\mu^{-\frac 12} x_\mu^{-1}), \quad \mu \to \infty,
\end{equation}
therefore
\begin{equation}\label{int.V.main}
\int_0^{x_{\sqrt\mu}}V  u^2 \, \dd x  =
\frac{\pi}{\mu^\frac12} \int_0^{x_{\sqrt\mu}} V  \, \dd x +  o(\mu^{-\frac12}), \quad \mu \to \infty.
\end{equation}

For any $\eps>0$, find $V_\eps \in C_0^\infty(\R)$ such that $\|V-V_\eps\|_{L^1(\R)}<\eps$. Then 
\begin{equation}\label{int.V.sin.1}
\left|\int_0^{x_{\sqrt\mu}}  V  \sin 2 \zeta \, \dd x \right| 
\leq \eps \|V\|_{L^1(\R)} + 
\left|  
\int_0^{x_{\sqrt\mu}}  V_\eps \sin 2 \zeta \, \dd x
\right|. 
\end{equation}
The integration by parts yields 
\begin{equation}\label{int.V.sin.2}
\begin{aligned}
\int_0^{x_{\sqrt\mu}}  V_\eps \sin 2 \zeta \, \dd x
&= - \frac{\cos (2\zeta(0) )V_\eps(0)}{2} \mu^{-\frac12}
\\
& \quad - \int_0^{x_{\sqrt\mu}} \frac{\cos 2\zeta}{2} \left(\frac{V_\eps'}{(\mu-Q)^\frac12} + \frac{V_\eps Q'} {2(\mu-Q)^\frac32}  \right) \; \dd x
\\
& = \BigO(\mu^{-\frac12}) + \BigO(a_\mu \mu^{-\frac32}) = \BigO(\mu^{-\frac12}), \quad \mu \to \infty,
\end{aligned}
\end{equation}
where we use that $V_\eps(x_{\sqrt \mu}) =0$ for all sufficiently large $\mu$ and \eqref{Q.asym} in the last step. Since $\eps>0$ is arbitrary, we conclude with \eqref{int.V.sin}.

Since $|R_1(\zeta)|<1/\zeta$, see \eqref{u^2.exp} and below, and $\zeta$ is decreasing, we have
\begin{equation}\label{int.V.R.1}
\int_0^{x_{\sqrt\mu}} | V R_1(\zeta)| \, \dd x \leq \frac{\|V\|_{L^1(\R)}}{\zeta(x_{\sqrt\mu})} 
\leq
\frac{\|V\|_{L^1(\R)}}{(\mu-\mu^\frac12)^\frac12 (x_\mu - x_{\sqrt\mu})}.
\end{equation}
By the mean value theorem,
\begin{equation}\label{int.V.R.2}
x_\mu - x_{\sqrt\mu} = Q^{-1}(\mu) - Q^{-1}(\mu^\frac12) \geq \frac{\mu - \mu^\frac12}{a_\mu},
\end{equation}
thus \eqref{int.V.R} follows from \eqref{int.V.R.1}, \eqref{int.V.R.2} and  \eqref{Q.asym}.

\noindent
$\bullet$ $x_{\sqrt \mu}<x <x_{\frac \mu 2}$: From \eqref{|u|.est} and $V \in L^1(\R)$, we have
\begin{equation}
\begin{aligned}
\int_{x_{\sqrt\mu}}^{x_\frac\mu 2 } |V| u^2 \; \dd x 
&\leq 
A_1^2 \int_{x_{\sqrt\mu}}^{x_\frac\mu 2 } \frac{|V| \, \dd x}{(\mu-Q)^\frac12} 
\\
& = \BigO(\mu^{-\frac12}) \int_{x_{\sqrt\mu}}^{x_\frac\mu 2 } |V| \, \dd x = o(\mu^{-\frac12}), \quad \mu \to \infty.
\end{aligned}
\end{equation}
\noindent
$\bullet$ $x_{\frac \mu 2}<x <x_{\frac 32 \mu}$: Since the support of $V_1$ is compact, only the integral with $V_2$ contributes for large $\mu$. Due to the controlled decay of $V_2$, see \eqref{V.decay}, we have
\begin{equation}
\int_{x_\frac\mu 2}^{x_{\frac 32 \mu}} |V| u^2 \; \dd x 
= \BigO(x_\mu^{-1-\eps})  \int_{x_\frac\mu 2}^{x_{\frac 32 \mu}} u^2 \; \dd x, \quad \mu \to \infty.
\end{equation}
The integral of $u^2$ is estimated in the proof of Proposition~\ref{prop:Q.Lq}, namely,
\begin{equation}\label{u2.int}
\left(\int_{x_\frac\mu 2}^{x_\mu-\delta} + \int_{x_\mu-\delta}^{x_\mu+\delta_1} + 
\int_{x_\mu+\delta_1}^{x_{\frac 32 \mu}}\right)  u^2 \; \dd x
= \BigO(a_\mu^{-1} \mu^\frac 12) + \BigO(a_\mu^{-\frac 23}),  \quad \mu \to \infty.
\end{equation}
Thus, we get from \eqref{Q.convex} that
\begin{equation}
\int_{x_\frac\mu 2}^{x_{\frac 32 \mu}} |V| u^2 \; \dd x = \BigO(\mu^{-\frac 12} x_\mu^{-\eps}), \quad \mu \to \infty.
\end{equation}

\noindent
$\bullet$ $x_{\frac 32 \mu}<x$: Again, there is no contribution of $V_1$ for large $\mu$ and \eqref{V.decay}, \eqref{|u|.est} and \eqref{zeta.32} yield (with some $C>0$, see also \eqref{u.q.32})
\begin{equation}
\int_{x_{\frac32 \mu} }^\infty |V| u^2 \; \dd x = \BigO(x_\mu^{-1-\eps} e^{-C \mu^\frac 12 x_\mu}), \quad  \mu \to \infty.
\end{equation}

Putting all these estimates together, we indeed get \eqref{Vu^2.asym}.
\end{proof}

\begin{remark}\label{rem:beta.inf}
For special choice $Q(x) = |x|^\beta$, $\beta \geq 2$, we obtain more precise asymptotics for $\{\mu_n\}$ from \eqref{EV.Tit}. Thus we can conclude further that, for $V$ as in \eqref{V.decay}, the eigenvalues $\{\lambda_n\}$ of $T$ satisfy as $n \to \infty$
\begin{equation}\label{la.n.beta.dec}
\lambda_n = \left( \frac{\pi}{\Omega_\beta} \left(n+\frac12 \right) \right)^\frac{2\beta}{\beta+2} + 
\frac{1}{\Omega_\beta'}
\left(
\frac{\pi}{\Omega_\beta} n
\right)^{-\frac{2}{\beta+2}} \int_\R V(x) \; \dd x + o \left(n^{-\frac{2}{\beta+2}}\right).
\end{equation}
When taking formally the limit $\beta \to + \infty$, the correction \eqref{lam.1.dec} due to $V$ becomes \eqref{SL.la.asym} with $l=1$; notice that the formula \eqref{SL.la.asym} is valid also for the perturbation of  $-\dd^2/\dd x^2$ with Dirichlet boundary conditions.
\end{remark}

\subsection{Perturbations by singular potentials}

Let $V \in W^{-s,2}(\R)$ with some $s \geq 0$, so
\begin{equation}\label{V.sing.delta}
\exists C>0, \ \exists s \geq 0, \ \forall \phi \in W^{1,2}(\R), \ |(V,\phi)| \leq C \|\phi\|_{W^{1,2}(\R)}^s \|\phi\|^{1-s},
\end{equation}
and define the form
\begin{equation}\label{bV.sing}
b_V(\phi,\psi):=(V,\phi\overline{\psi}), \quad \phi, \psi \in \Dom(a).
\end{equation}
An extension of Theorem~\ref{thm:aho.Lptau} for singular potentials is Theorem~\ref{thm:aho.sing} below, where sufficient conditions on $V$ so that $b_V$ satisfies \eqref{asm:b} are stated. In the proof, the following estimates of $L^q$-norms of $\{\psi_k'\}$ are used.
\begin{lemma}
Let $Q$ satisfy Assumption \ref{asm:Q} with \eqref{Q.beta} replaced by \eqref{Q.unbd}, let $\{\mu_k\}$ be the eigenvalues of $A$ and let $\{\psi_k\}$ be as in \eqref{psi.u.rel}. Then 
\begin{equation}\label{psi_k'.q}
\|\psi'_k\|_{L^q(\R)} = \BigO \left( \mu_k^\frac 12 \|\psi_k\|_{L^q(\R)} \right), \quad k \to \infty.
\end{equation}
\end{lemma}
\begin{proof}
By the Gagliardo-Nirenberg interpolation inequality with $\alpha=1/2$, see~\cite[Lecture~II]{Nirenberg-1959-13}, and since $\psi_k$ is an eigenfunction of $A$, we get
\begin{equation}\label{psi'.GN}
\begin{aligned}
\|\psi'_k\|_{L^q(\R)} 
&\leq C \|\psi_k''\|_{L^q(\R)}^\frac12 \|\psi_k\|_{L^q(\R)}^\frac12 
\\
&\leq C \left(\mu_k^\frac12 \|\psi_k\|_{L^q(\R)}^\frac12  + \|Q^\frac12\psi_k\|_{L^q(\R)}^\frac12 \right) \|\psi_k\|_{L^q(\R)}^\frac12.
\end{aligned}
\end{equation}
Finally, $Q^\frac{1}{2}$ satisfies the conditions on the weight $w$ in Proposition~\ref{prop:Q.Lq} and so \eqref{psi_k'.q} follows from~\eqref{psi'.GN} and \eqref{psik.w.q} with $w = Q^\frac12$.
\end{proof}

\begin{theorem}\label{thm:aho.sing}
Let $Q$ satisfy Assumption~\ref{asm:Q} and let $A$ be the Schr\"odinger operator from~\eqref{aho.def}. 
If $V \in W^{-s,2}(\R)$ with 
\begin{equation}\label{beta.sing}
0 \leq s< \frac{\beta-1}{2\beta},
\end{equation}
then $A$ and the form $b_V$ from~\eqref{bV.sing} satisfy the condition \eqref{asm:b} with 
\begin{equation}\label{V.sing.alpha}
\gamma = \frac{2 \beta}{\beta+2}, 
\qquad
\alpha= \frac{1-2\beta s}{2(\beta +2)}.
\end{equation}
If the support of $V$ is compact and 
\begin{equation}\label{V.sing.s.comp}
0 \leq s< \frac 12,
\end{equation}
then $A$ and the form $b_V$ from~\eqref{bV.sing} satisfy the condition \eqref{asm:b} with 
\begin{equation}\label{V.sing.alpha.comp}
\gamma = \frac{2 \beta}{\beta+2}, 
\qquad
\alpha= \frac{1-\beta s}{\beta +2}.
\end{equation}
\end{theorem}
\begin{proof}
The letter $C$ denotes a constant, which can vary in every step, however, it is independent of $m$ and $n$.
We employ~\eqref{V.sing.delta}, H\"older inequality, \eqref{psi_k'.q} with $q=4$ and finally \eqref{Q.beta.EV.asym}, \eqref{aho.ef.est}
\begin{equation}\label{bV.sing.alpha}
\begin{aligned}
|b_V(\psi_m,\psi_n)| 
& \leq C 
\left( 
\|\psi_m' \psi_n\|^s +  \|\psi_n' \psi_m\|^s + \|\psi_m \psi_n\|^s
\right) \|\psi_m\psi_n\|^{1-s}.
\\
& \leq C
\left( \mu_m^\frac s 2 +  \mu_n^\frac s 2 + 1
\right)\|\psi_m\|_{L^4(\R)}  \|\psi_n\|_{L^4(\R)}
\\
& \leq C
\left( \mu_m \mu_n\right)^\frac s 2 \|\psi_m\|_{L^4(\R)}  \|\psi_n\|_{L^4(\R)}
\\
& \leq C
(mn)^\frac{2 \beta s -1}{2(\beta +2)} \log m \log n.
\end{aligned}
\end{equation}
Thus the condition~\eqref{asm:b} is satisfied with $\alpha$ as in \eqref{V.sing.alpha}.

If the support of $V$ is compact, we use the pointwise bounds for $|\psi_k|$ and $|\psi'_k|$, see \eqref{|psi|}, \eqref{|psi'|}. Similarly as above, we obtain from \eqref{V.sing.delta} that
\begin{equation}\label{bV.sing.alpha.comp}
|b_V(\psi_m,\psi_n)| 
\leq C (mn)^\frac{\beta s -1}{\beta+2}
\end{equation}
and so the condition~\eqref{asm:b} is satisfied with $\alpha$ as in \eqref{V.sing.alpha.comp}.
\end{proof}

Notice that for $\beta \to + \infty$ in \eqref{beta.sing}, we recover the condition $s <1/2$ derived in Section~\ref{subsubsec:Lapl.sing} for perturbations of $-\dd^2/\dd x^2$ on finite interval.

\begin{corollary}\label{cor:aho.sing}
Let $A$ be as in \eqref{aho.def} and $V \in W^{-s,2}(\R)$ with $s$ satisfying \eqref{beta.sing} or \eqref{V.sing.s.comp} if the support of $V$ is compact. Then the eigensystem of $T$, being the form sum of these $A$ and $V$, contains a Riesz basis.
\end{corollary}

Finally, we consider the perturbation of $A$ by $\delta$-interactions like in Section~\ref{subsubsec:Lapl.delta}, see also \eqref{A.V.intro}. Let
\begin{equation}\label{b.delta.inf.beta}
b_\delta^\infty[\psi] := \sum_{k \in \Z} \nu_k |\psi(x_k)|^2, \quad \{\nu_k\} \in \ell^1(\Z), \quad \{x_k\} \subset \R, \quad \psi \in \Dom(a).
\end{equation}
It follows from \eqref{sup|psi|} that
\begin{equation}
|b_\delta^\infty(\psi_m,\psi_n)| \leq C \|\nu\|_{\ell^1(\Z)} (mn)^\frac{\beta-4}{6(\beta+2)},
\end{equation}
thus $A$ and $b_\delta^\infty$ satisfy the condition~\eqref{asm:b} if $\beta >1$. Hence, 
by Theorem \ref{thm:RB}, the eigensystem of $T$, being the form sum of $A$ in \eqref{aho.def} and $b_\delta^\infty$ in \eqref{b.delta.inf.beta}, contains a Riesz basis.

{\footnotesize
\bibliographystyle{acm}
\bibliography{C:/Data/00Synchronized/references}
}

\end{document}